\documentclass[12pt]{iopart}

\newcommand{\eqref}[1]{equation (\ref{#1})}
\usepackage{graphicx}
\usepackage{amsthm, amssymb}
\theoremstyle{plain}

\newtheorem{remark}{Remark}
\newtheorem{lemma}{Lemma}
\theoremstyle{remark}

\usepackage{algorithm}
\usepackage{algorithmic}

\usepackage{url}
\begin{document}

\title[A data-driven approach to 1D inverse scattering]{A data-driven approach to solving a 1D inverse scattering problem}

\author{Tristan van Leeuwen$^{1,2}$ and Andreas Tataris$^{1}$}

\address{
$^{1}$Utrecht University, Utrecht, the Netherlands. \\
$^{2}$Centrum Wiskunde \& Informatica, Amsterdam , the Netherlands.
}
\ead{T.van.Leeuwen@cwi.nl}
\vspace{10pt}

\begin{abstract}
In this paper, we extend the ROM-based approach for inverse scattering with Neumann boundary conditions, introduced by Druskin at. al. (Inverse Problems 37, 2021), to the 1D Schr{\"o}dinger equation with impedance (Robin) boundary conditions. We also propose a novel data-assimilation (DA) inversion method based on the ROM approach, thereby avoiding the need for a Lanczos-orthogonalization (LO) step. Furthermore, we present a detailed numerical study and comparison of the accuracy and stability of the DA and LO methods. 
\end{abstract}

%
\vspace{2pc}

%
%
%

\section{Introduction\label{Introduction }}

Inverse scattering appears in many applications, including medical imaging, non-destructive testing, and geophysical exploration \cite{pike2001scattering}. While acquisition setups differ, at their core all these inverse problems involve a wave-equation and require estimation of its variable coefficients from boundary data. Approaches to solving the resulting non-linear inverse problem can be classified as either \emph{direct} or \emph{indirect} methods. The direct methods originate in classical inverse scattering theory and rely on formulating a linear relation between scattering data and the medium parameters, see e.g. \cite{Ware1969}. The indirect methods formulate a non-linear data-fitting problem that can be solved iteratively \cite{Tarantola1984}.

The direct methods have recently attracted renewed attention, in particular in the geophysical community \cite{Broggini2012}. A recent development is the use of data-driven reduced-order models for solving the inverse problem \cite{Druskin2021LippmannSchwingerLanczosAF}. We summarize this procedure below.

\subsection{Approach}
The state equation is denoted as
\[
\left(A_q + k^2 I\right)u(k) = s,
\]
with $u$ denoting the state for wavenumber $k$, $s$ the source term, and $q$ the variable coefficient included in the differential operator $A_q.$ The measurements are given by $f_i = \langle s, u_i\rangle = \left\langle s, \left(A_q + k_i^2 I\right)^{-1}s\right\rangle$ for $i = 0, 1, \ldots, m-1$. The approach is to first estimate the states $u_i$ from the measurements, and subsequently estimate $q$ from these using the state equation.

The first step of estimating the states is approached via a reduced-order model which looks for a solution of the state equation in $\mathcal{U} = \mathrm{span}\left(\{u_i\}_{i=0}^{m-1}\right)$ by projecting the state equation on this subspace. This requires computing $\langle u_i, u_j\rangle$ and $\langle u_i, A_q u_j\rangle$. Remarkably, this can be done directly in terms of the measurements, without explicit reference to the states $u_i$. To approximate the states, then, we solve the projected state equation and represent the solution in a basis $\mathcal{U}^{(0)}$ of solutions $u_i^{(0)}$ for a given $q_0$. This last step is intricate and requires a Lanczos orthogonalization, see \cite{Druskin2021LippmannSchwingerLanczosAF} for more details.

The next step of retrieving $q$ from the approximated states, $\widetilde{u}_i$, can be approached in different ways. We can follow an equation error approach (see e.g. \cite{karkkainen1997equation}) and solve $q$ from 
\[
\left(A_q + k_i^2 I\right)\widetilde{u}_i = s.
\]
Alternatively, we can solve it from a Lipmann-Schwinger integral equation (see, e.g. \cite{kouri2003inverse})
\[
f_i - f_i^{(0)} = - \left\langle u_i^{(0)}, \left(A_q - A_0\right)\widetilde{u}_i\right\rangle.
\]

\subsection{Contributions and outline}
The ROM-based approach has been applied in various settings, including time domain wave propagation, see e.g. \cite{Borcea2021ROM} and frequency-domain diffusion processes, see \cite{Druskin2021LippmannSchwingerLanczosAF}. As a first step towards extending this procedure to frequency-domain wave-problems, we extend the approach to a 1D Schrödinger equation with impedance boundary conditions. It turns out that both reflection and transmission measurements are needed to compute the ROM matrices from the data. Furthermore, we propose an alternative approach to the Lanczos-based state estimation approach described by \cite{Druskin2021LippmannSchwingerLanczosAF}. To study the accuracy and stability properties of the resulting methods, we present numerical experiments.

The paper is organized as follows. First, we review the forward problem and present the relations between the boundary data and required ROM matrices. Then, we discuss the two-step approach to solve the inverse problem; state estimation and subsequent estimation of the scattering potential from the state. We then present numerical experiments to illustrate the accuracy and stability of both methods on noisy data. We conclude the paper with a brief summary of the main findings and discussion on further work.

\section{The forward problem\label{The forward problem}}
Consider a Schrödinger equation
\begin{equation}
\label{eq:strong_formulation}
-u''(x;k) + q(x)u(x;k) - k^2u(x;k) = 0, \quad x \in (0,1)
\end{equation}
with boundary conditions
\begin{equation}
u'(0;k) + \imath k u(0;k) = 2\imath k, u'(1;k) - \imath k u(1;k) = 0,
\end{equation}
which corresponds to an incoming plane wave from $-\infty$. The scattering potential is assumed to have compact support in $(0,1)$. The measurements are given by 
\begin{equation}
\label{eq:boundary_data}
f(k) = u(0;k), \quad g(k) = u(1;k).
\end{equation}
Well-posedness of this forward problem has been well-established (at least when $q$ is continuous),
since the boundary value problem can be transformed to the Lippmann-Schwinger integral. Then it is sufficient to study just the integral equation see e.g. \cite{kresslinear}.

\subsection{A reduced-order model}
The point of departure for the ROM-based approach is the weak formulation of \eqref{eq:strong_formulation}
\begin{equation}
\langle u', \phi'\rangle + \langle qu,\phi\rangle - k^2 \langle u, \phi \rangle -\imath k \left(f(k)\overline{\phi(0)} + g(k)\overline{\phi(1)}\right) = -2\imath k \overline{\phi(0)},
\end{equation}
where $\langle \cdot, \cdot\rangle$ denotes the standard inner product in $L^2(0,1)$ and $\overline{\cdot}$ denotes complex conjugation.

Using a basis of solutions $\{u_i\}_{i=0}^{m-1}$ with $u_i \equiv u(\cdot;k_i)$, the resulting system matrices are defined correspondingly
\begin{eqnarray}
\label{eq:rom_s}
S_{ij} = \langle u_j', u_i'\rangle + \langle qu_j,u_i\rangle,\\
\label{eq:rom_m}
M_{ij} = \langle u_j,u_i\rangle,\\
B_{ij} = f_j\overline{f_i} + g_j\overline{g_i},
\end{eqnarray}
and right-hand-side
\begin{equation}
b_i = -2\imath k \overline{f_i}.
\end{equation}
The main feature making this approach useful for solving the inverse problem is that the system matrices can be computed from the data directly, as per the following Lemma.
\begin{lemma}
\label{lma:rommatrices}
The ROM system matrices $S, M$ (equations (\ref{eq:rom_s}) and (\ref{eq:rom_m})) are given in terms of the boundary data $\{f_i\}_{i=0}^{m-1}$ and $\{g_i\}_{i=0}^{m-1}$ (equation (\ref{eq:boundary_data})) as 
\[
S_{ij} = \imath \left(\frac{k_i k_j B_{ij}}{k_i - k_j} -2 \frac{k_j^2k_i f_j + k_i^2k_j \overline{f}_i}{k_i^2 - k_j^2}\right), \quad i\not=j
\]
\[
S_{ii} = k_i^2\left(\Re(f_i)\Im(f_i') - \Im(f_i)\Re(f_i') + \Re(g_i)\Im(g_i') - \Im(g_i)\Re(g_i')\right)- \Im(f_i') - \Im(f_i)/k_i.
\]
\[
M_{ij} = \imath \left(\frac{B_{ij}}{k_i - k_j} - 2\frac{k_i {f}_j + k_j \overline{f}_i}{k_i^2 - k_j^2}\right), \quad i\not=j
\]
\[
M_{ii} = \Re(f_i)\Im(f_i') - \Im(f_i)\Re(f_i') + \Re(g_i)\Im(g_i') - \Im(g_i)\Re(g_i') - \Im(f_i') + \Im(f_i)/k_i.
\]
\end{lemma}
The proof of this Lemma can be found in the appendix.

Correspondingly, the approximate solution is then given by 
\begin{equation}
\label{eq:urom}
\widetilde{u}(x;k) = \sum_{i=0}^{m-1} c_i(k) u_i(x),
\end{equation}
with
\[
\left(S - k^2 M - \imath k B\right)\mathbf{c}(k) = \mathbf{b}(k).
\]
\begin{remark}
\label{rmk:1}
From the proof of Lemma \ref{lma:rommatrices} we see that $c_i(k_j) = \delta_{ij}$. Thus $\widetilde{u}$ will match the boundary data.
\end{remark}
We refer the reader to \cite{Fink1983OnTE,Maday2002APC,Veroy2003APE,Sen2006NaturalNA} for the discussion regarding the approximation error of such ROM-approximations.

\section{The inverse problem\label{The inverse problem}}
The inverse problem is now to retrieve $q$ from boundary measurements at wave numbers $\{k_i\}_{i=0}^{m-1}$. As outlined in the introduction, this is achieved in a 2-step procedure. First the states $\{u_i\}_{i=0}^{m-1}$ are estimated from the data, and subsequently the scattering potential is estimated from these approximated states.

\subsection{Estimating the state}
As outlined in the previous section, we can compute the \emph{coefficients} in \eqref{eq:urom} directly from the data following the ROM-based approach. Since the basis $\{u_i\}_{i=0}^{m-1}$ needed to evaluate \eqref{eq:urom} is unknown, however, we need to use a different basis. The basic idea is to use states $\{u^{(0)}_i\}_{i=0}^{m-1}$ corresponding to a given $q^{(0)}$ instead. It is tempting to directly replace  \eqref{eq:urom} by
\[
\widetilde{u}(x;k) = \sum_{i=0}^{m-1} c_i(k) u_i^{(0)}(x),
\]
however, this will not work as it would yield $\widetilde{u}(x;k_i)  = u_i^{(0)}(x)$, see Remark \ref{rmk:1}. Below, we discuss two alternatives.

\subsubsection{Lanczos orthogonalization}
\label{section:Lanczos}
The authors of \cite{Druskin2021LippmannSchwingerLanczosAF} propose to use an orthogonalization procedure as follows. They first apply the $M$-orthogonal Lanczos procedure to $M^{-1}S$, which yields matrices $Q\in\mathbb{C}^{m\times r}$ and $T \in \mathbb{C}^{r \times r}$, where $r\leq m$, satisfying
\[
Q^*SQ = T, \quad Q^*MQ = I.
\]
The ROM-approximation of the state is then given by 
\begin{equation}
\label{eq:urom2}
\widetilde{u}(x;k) = \sum_{i=0}^{m-1} c_i(k) v_i(x),
\end{equation}
with $\mathbf{c}$ satisfying
\[
\left(T - k^2 I - \imath k Q^*BQ\right)\mathbf{c}(k) = Q^*\mathbf{b}(k),
\]
and $\{v_i\}_{i=0}^{r-1}$ an \emph{orthogonal} basis w.r.t. the regular $L^2$-inner product defined as
\[
v_j = \sum_{i=0}^{m-1}Q_{ij}u_i.
\]
The expression in \eqref{eq:urom2} is equivalent to \eqref{eq:urom} (although the coefficients differ). Because we do not have access to the states $\{u_i\}_{i=0}^{m-1}$, and cannot form the orthogonal basis $\{v_i\}_{i=0}^{r-1}$, we replace it by $\{v_i^{(0)}\}_{i=0}^{r-1}$, obtained as
\[
v_j^{(0)} = \sum_{i=0}^{m-1}Q_{ij}^{(0)}u_i^{(0)},
\]
where the states $u_i^{(0)}$ are the solutions for a reference scattering potential $q^{(0)}$ and $Q^{(0)}$ is obtained by applying the Lanczos procedure to the corresponding system matrices.
\begin{remark}
\label{rmk:2}
In practice, we replace $M$ by $M+\epsilon I$ for some $\epsilon>0$ to ensure it is invertible and to stabilize the Lanczos procedure.
\end{remark}
\subsubsection{Data-assimilation}
\label{section:DA}
An alternative approach is inspired by \cite{van2015penalty} and sets up an overdetermined system of equations which ensures that the resulting estimate of the internal solution closely matches the data. We directly define the approximated state in terms of the reference solutions
\[
\widetilde{u}(x;k) = \sum_{i=0}^{m-1} c_i(k) u_i^{(0)}(x),
\]
where the coefficients $\mathbf{c}(k)$ are obtained by solving the following least-squares problem
\begin{equation}
\label{eq:DA}
\min_{\mathbf{c}}\left\|\left(\begin{array}{c}
S - k^2 M - \imath k B \\
\rho{\mathbf{f}^{(0)}}^T \\
\rho {\mathbf{g}^{(0)}}^T 
\end{array}\right)
\mathbf{c} - 
\left(\begin{array}{c}
\mathbf{b}(k) \\
\rho f(k)\\
\rho g(k)
\end{array}\right)\right\|_2,
\end{equation}
where $\rho > 0$ is a penalty parameter controlling the trade-off between data-fit and model-fit. The required data $f(k)$ and $g(k)$ can be obtained by solving \eqref{eq:urom} and using the coefficients to interpolate them.

\subsection{Estimating the scattering potential}
Using the weak formulation of the differential equation we obtain a Lippmann-Schwinger-type equation,
\begin{equation}
\label{eq:LS1}
f(k) - f^{(0)}(k) = - \frac{1}{2\imath k}\int_0^1 u^{(0)}(x;k) u(x;k) (q(x) - q_0(x)) \mathrm{d}x.
\end{equation}
Representing $q$ in terms of a suitable basis and enforcing the equation for wavenumbers $\{k_i\}_{i=0}^{m-1}$ yields a system of equations. In practice, we replace $u$ by its approximation $\widetilde{u}$ and solve it in a least-squares sense to obtain an estimate of $q$:
\begin{equation}
\label{eq:LS}
\min_{\mathbf{q}} \|K\mathbf{q} - (\mathbf{f}-\mathbf{f}^{(0)})\|_2^2 + \alpha \|\mathbf{q}\|_2^2.
\end{equation}
\begin{remark}
Note that replacing $u$ by $\widetilde{u}$ in \eqref{eq:LS1} induces an error in $K$. To explicitly account for this, a Total Least-Squares (TLS) formulation (see e.g. \cite{tataris2022regularised} for its use in inverse scattering) might be beneficial. 
\end{remark}
\section{Numerical results\label{Numerical results}} 

The inversion procedure consists of two steps; state estimation and estimation of the scattering potential from the states. For the first step, we use either the Lanczos orthogonalization approach (LO) with parameter $\epsilon$, or the data-assimilation approach (DA) with parameter $\rho$. With the approximated states, the scattering potential is then estimated by solving the regularized Lippmann-Schwinger equation, with parameter $\alpha$. This two-step algorithm is outline in Algorithm \ref{alg}. Implementation of the described method is fairly straightforward. The code used to produce these results is available at \url{https://github.com/ucsi-consortium/1DInverseScatteringROM}.

\begin{algorithm}
\caption{Overview of the two-step inversion procedure to estimate the states and scattering potential from boundary data.}
\label{alg}
\begin{algorithmic}
    \REQUIRE{reference $q^{(0)}$, data $f, g$ at wavenumbers $\{k_i\}_{i=0}^{m-1}$, regularisation parameters ($(\epsilon,\alpha)$ or $(\rho, \alpha)$)}
    \ENSURE{reconstructed states $\{\widetilde{u}_i\}_{i=0}^{m-1}$ and scattering potential $\widetilde{q}$.}
    \STATE{}
    \STATE{\textit{Step 1: state estimation}}
    \STATE{\hspace{1cm}Compute ROM-matrices $M, S, B$ according to Lemma \ref{lma:rommatrices}}
    \STATE{\hspace{1cm}Compute reference states $\{u_i^{(0)}\}_{i=0}^{m-1}$ corresponding to $q^{(0)}$.}
    \STATE{\hspace{1cm}Compute approximate states $\{\widetilde{u}_i\}_{i=0}^{m-1}$ at wavenumbers $\{k_i\}_{i=0}^{m-1}$ according to the LO or DA procedures (outlined in sections 3.1.1, 3.1.2 resp.)}
    \STATE{}
    \STATE{\textit{Step 2: estimating the scattering potential}}
    \STATE{\hspace{1cm}Reconstruct the scattering potential $
    \widetilde{q}$ according to the procedure outlined in section 3.1.3}
\end{algorithmic}

\end{algorithm}
\clearpage
\subsection{Experimental settings}
To illustrate the methods, we use the scattering potential depicted in figure \ref{qfg}. The data are obtained by numerically solving the Schr{\"o}dinger equation for $m=10$ equispaced wave numbers in the interval $(0,10)$.

\begin{figure}
    \centering
    \includegraphics[scale=.5]{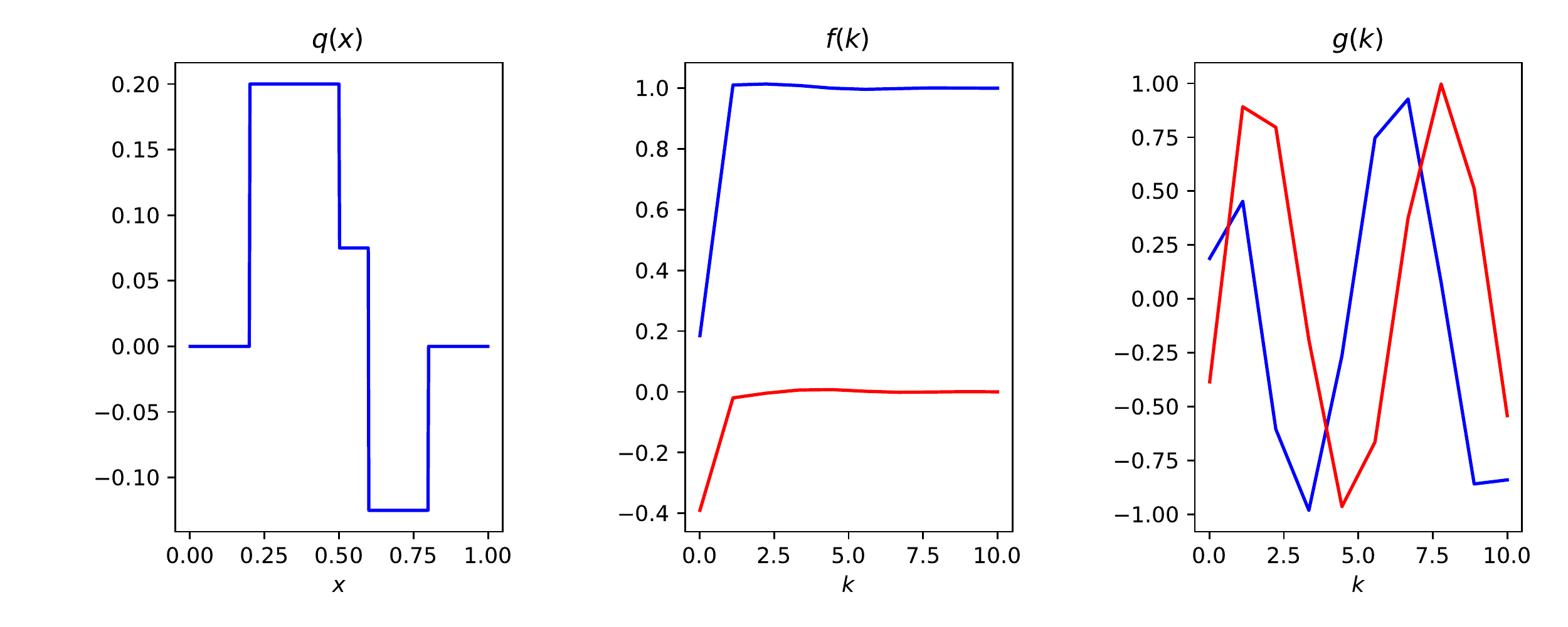}
    \caption{From left to right, the scattering potential $q$, the real (blue) and imaginary (red) part of the reflection data, $f,$ and the real and imaginary part of the transmission data, $g,$.}
    \label{qfg}
\end{figure}
\clearpage
\subsection{Benchmark results}
As a benchmark, we reconstruct the scattering potential using the approach described in section 3.1.3 using the true states (as the ideal setting) and the reference states for $q^{(0)} = 0$ (which corresponds to the Born approximation). The results are shown in figures \ref{rec_true} and \ref{rec_back}. Even using the true states we do not get a perfect reconstruction of the scattering potential due to the band-limited nature of the data. Furthermore, the inferior result obtained using the Born approximation underlines the need for non-linear inversion.
\clearpage
\begin{figure}
    \centering
    \includegraphics{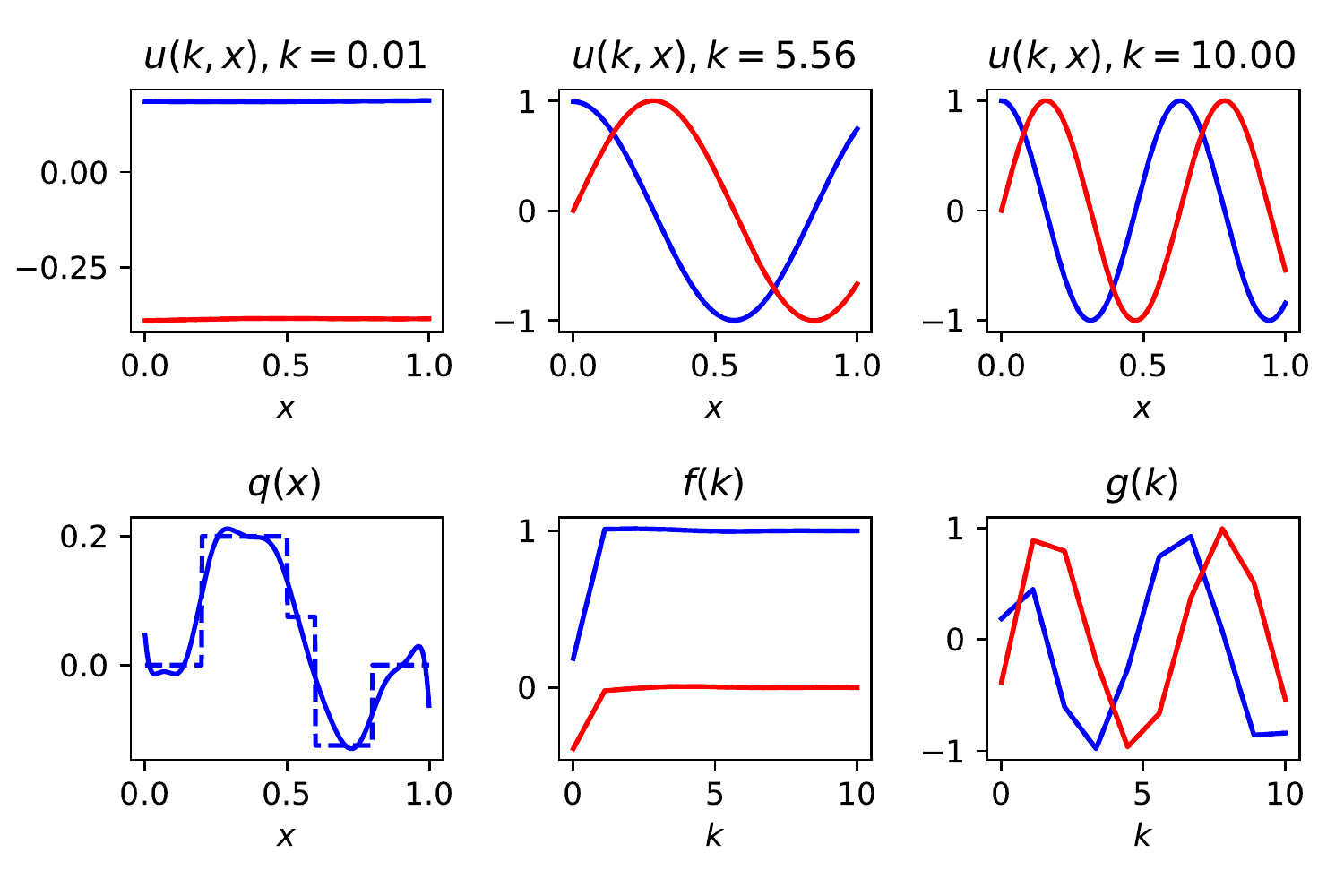}
    \caption{Results using the true state to reconstruct the scattering potential. The top row shows the (reconstructed) states (solid) used in the subsequent step to estimate the scattering potential as well as the true states (dashed). In the second row we see the reconstructed scattering potential (solid) and the corresponding data. The real part of the quantities is shown in blue, while the imaginary part is shown in red.}
    \label{rec_true}
\end{figure}
\begin{figure}
    \centering
    \includegraphics{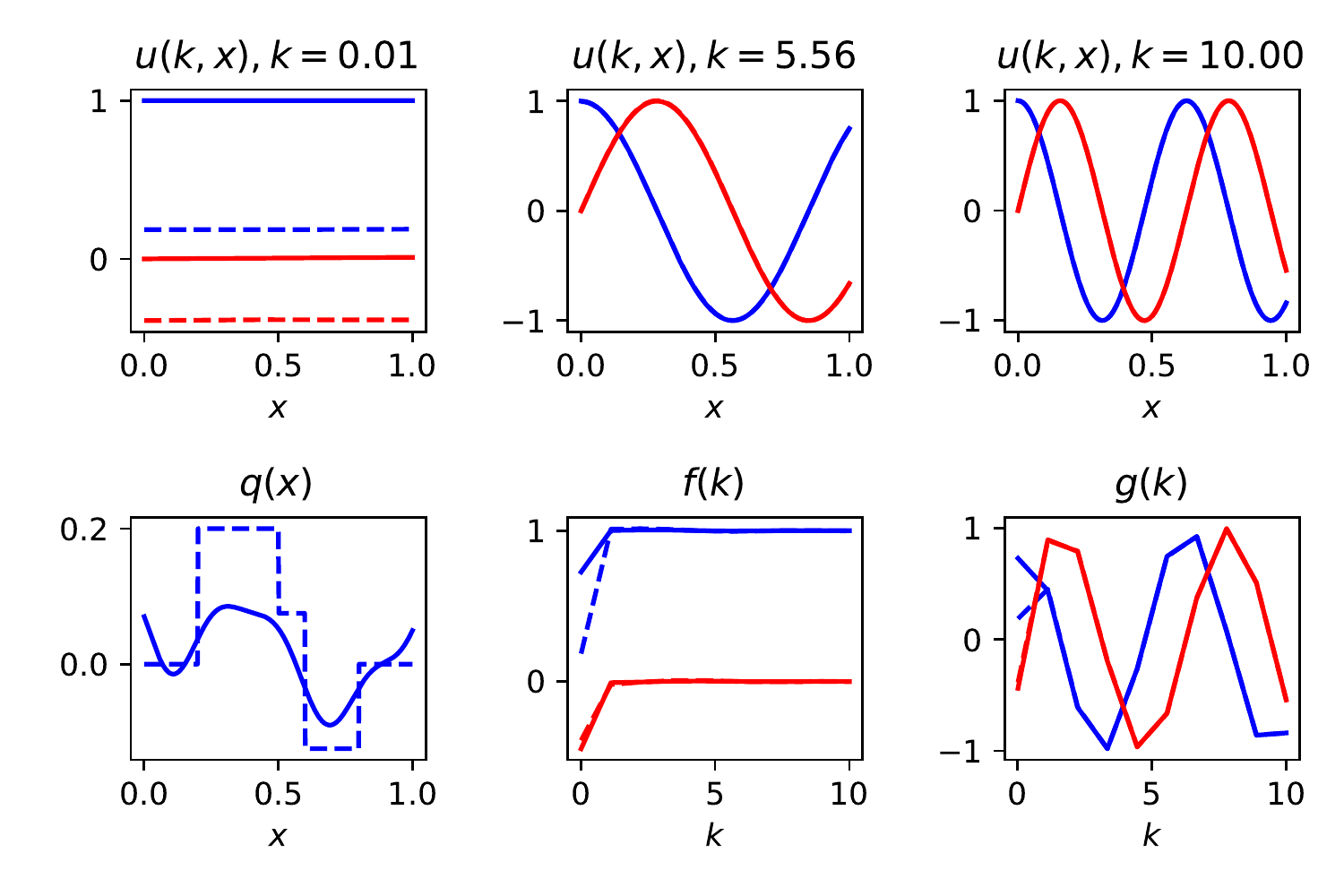}
    \caption{Results using the reference state to reconstruct the scattering potential (i.e., the Born approximation). The top row shows the (reconstructed) states  (solid) used in the subsequent step to estimate the scattering potential as well as the true states (dashed). In the second row we see the reconstructed scattering potential (solid) and the corresponding data. The real part of the quantities is shown in blue, while the imaginary part is shown in red.}
    \label{rec_back}
\end{figure}

\clearpage

\subsection{Noiseless data}
Next, we present the results yielded by the (LO) and (DA) methods for noise-free data in figures \ref{rec_Lan}, \ref{rec_DA} respectively. We observe that the DA method gives slightly more accurate reconstructions of the states. The corresponding reconstructed scattering potentials are slightly different, but there seems to be little difference in the accuracy of the reconstructions.
\clearpage
\begin{figure}
    \centering
    \includegraphics{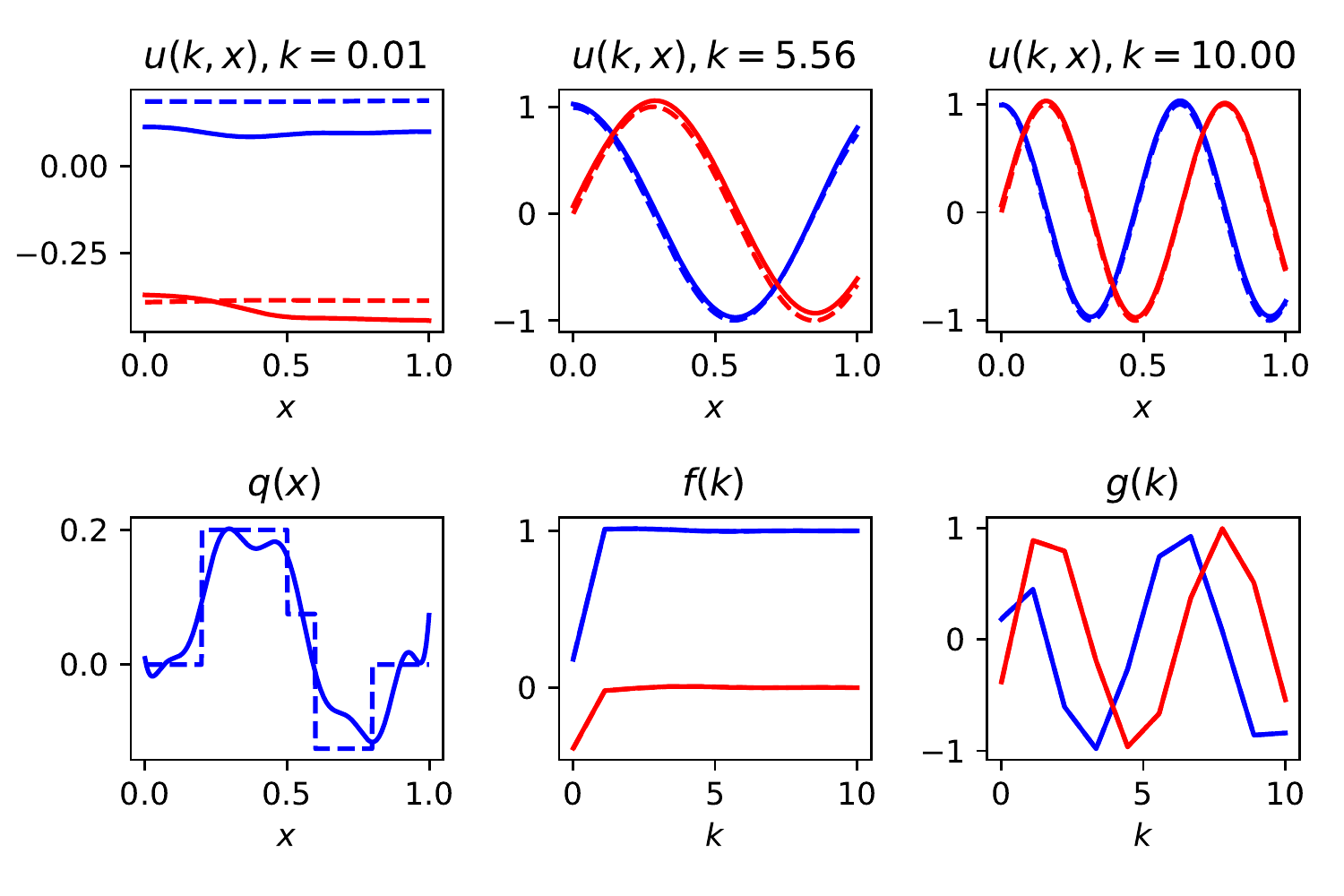}
    \caption{Results using LO-approach on noiseless data. The top row shows the (reconstructed) states (solid) used in the subsequent step to estimate the scattering potential as well as the true states (dashed). In the second row we see the reconstructed scattering potential (solid) and the corresponding data. The real part of the quantities is shown in blue, while the imaginary part is shown in red.}
    \label{rec_Lan}
\end{figure}
\begin{figure}
    \centering
    \includegraphics{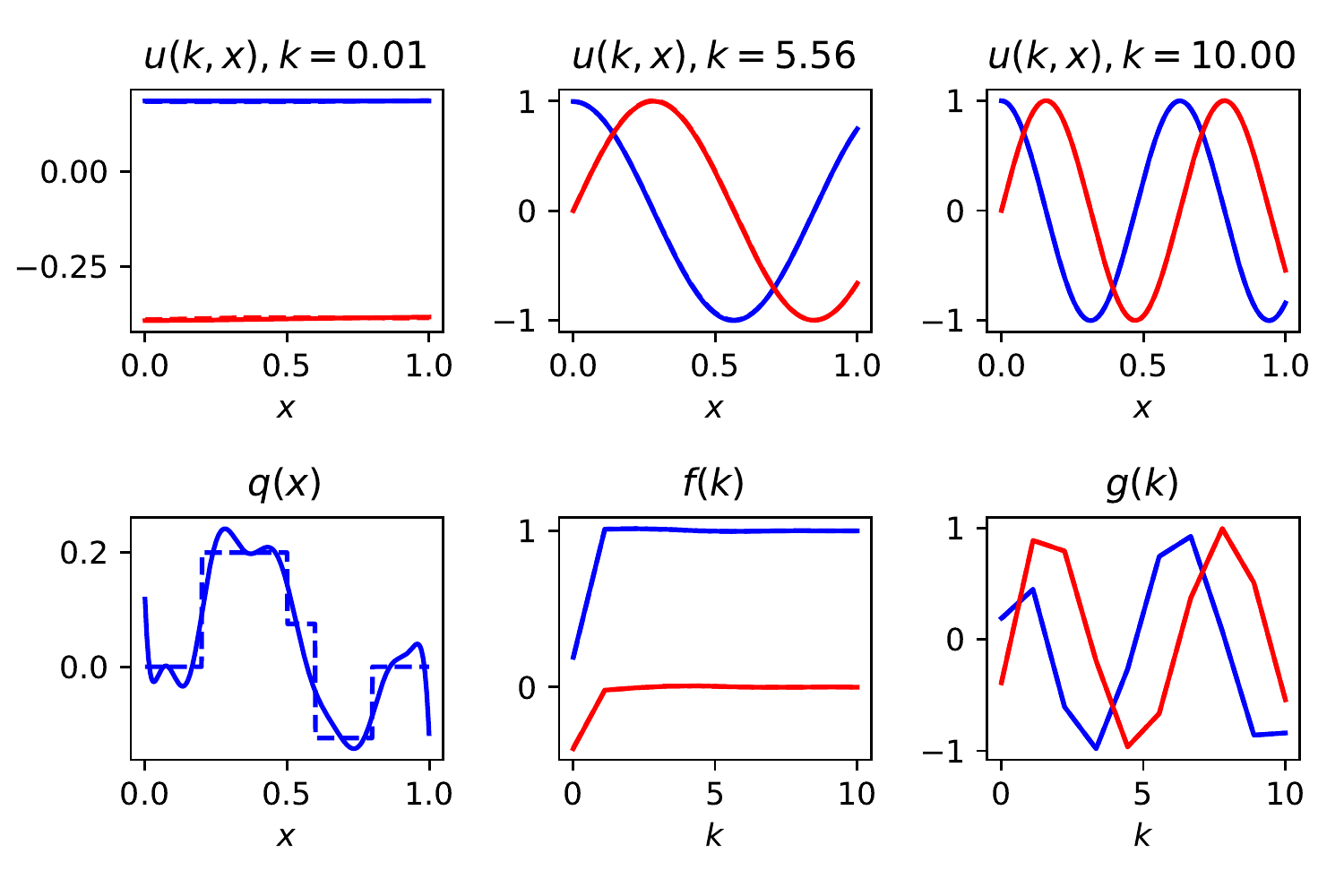}
    \caption{Results using DA-approach on noiseless data. The top row shows the (reconstructed) states (solid) used in the subsequent step to estimate the scattering potential as well as the true states (dashed). In the second row we see the reconstructed scattering potential (solid) and the corresponding data. The real part of the quantities is shown in blue, while the imaginary part is shown in red.}
    \label{rec_DA}
\end{figure}
\clearpage
\subsection{Noisy data}
In this subsection we compare the methods on noisy data. In particular, we add i.d.d. normally distributed noise to the data with mean zero and variance $\sigma^2$. The parameters $\epsilon, \rho, \alpha$ are chosen to yield the best approximation (as measured by the $L^2$ error between the reconstructions and the ground-truth, averaged over 100 realizations of the noise). The corresponding plots showing the dependence of the error on the parameters are included in the appendix. In table \ref{tabone} we summarize the results for varying $\sigma$. The corresponding plots are shown in figure \ref{fig:noisyresults}. As expected, the noise influences the reconstruction of the state and consequently the reconstruction of the scattering potential. Overall, we see that the DA method gives superior estimates of the state. In terms of the scattering potential there is no significant difference between both methods, however, for moderate noise levels the DA method gives more stable results with a much smaller variance in the error.
\clearpage
\begin{table}
\caption{Comparison between the relative errors in reconstructed states and scattering potential for both methods. We report the average and standard deviation over 100 realizations of the noise.}
\label{tabone}
\begin{indented}
\lineup
\item[]\begin{tabular}{@{}*{5}{l}}
\br                              
$\sigma$&method&parameters& error in $u$&error in $q$\cr
\mr
$10^{-6}$ & LO$(\epsilon, \alpha)$ & $(10^{-3},10^{-3})$ & $1.5\cdot 10^{-1}\,(1.6\cdot 10^{-3})$ & $4.7\cdot 10^{-1}\,(3.2\cdot 10^{-3})$
\cr \mr
          & DA$(\rho, \alpha)$ & $(10^{-2},10^{-4})$ & $6.1\cdot 10^{-3}\,(1.4\cdot 10^{-5})$ & $3.9\cdot 10^{-1}\,(2.3\cdot 10^{-3})$
\cr \mr
$10^{-5}$ & LO$(\epsilon, \alpha)$ & $(10^{-2},10^{-3})$ & $1.5\cdot 10^{-1}\,(5.3\cdot 10^{-4})$ & $4.6\cdot 10^{-1}\,(2.3\cdot 10^{-3})$
\cr \mr
          & DA$(\rho, \alpha)$ & $(10^{-1},10^{-3})$ & $6.1\cdot 10^{-3}\,(3.0\cdot 10^{-5})$ & $4.5\cdot 10^{-1}\,(2.8\cdot 10^{-3})$
\cr \mr
$10^{-4}$ & LO$(\epsilon, \alpha)$ & $(10^{-2},10^{-2})$ & $1.8\cdot 10^{-1}\,(1.5\cdot 10^{-1})$ & $5.7\cdot 10^{-1}\,(1.4\cdot 10^{-1})$
\cr \mr
          & DA$(\rho, \alpha)$ & $(10^{-1},10^{-2})$ & $6.2\cdot 10^{-3}\,(3.4\cdot 10^{-4})$ & $5.3\cdot 10^{-1}\,(3.2\cdot 10^{-3})$
\cr \mr
$10^{-3}$ & LO$(\epsilon, \alpha)$ & $(10^{-1},10^{-2})$ & $2.1\cdot 10^{-1}\,(1.2\cdot 10^{-1})$ & $6.2\cdot 10^{-1}\,(1.3\cdot 10^{-1})$
\cr \mr
          & DA$(\rho, \alpha)$ & $(10^{ 0},10^{-2})$ & $6.4\cdot 10^{-3}\,(7.0\cdot 10^{-4})$ & $6.0\cdot 10^{-1}\,(5.9\cdot 10^{-2})$
\cr \mr
$10^{-2}$ & LO$(\epsilon, \alpha)$ & $(10^{-1},10^{-1})$ & $2.6\cdot 10^{-1}\,(7.1\cdot 10^{-1})$ & $9.2\cdot 10^{-1}\,(9.4\cdot 10^{-2})$
\cr \mr
          & DA$(\rho, \alpha)$ & $(10^{ 1},10^{-1})$ & $1.4\cdot 10^{-2}\,(4.4\cdot 10^{-3})$ & $9.2\cdot 10^{-1}\,(9.0\cdot 10^{-2})$
\cr\br
\end{tabular}
\end{indented}
\end{table}
\clearpage
\begin{figure}
\centering
\begin{tabular}{cc}
\includegraphics[scale=.4]{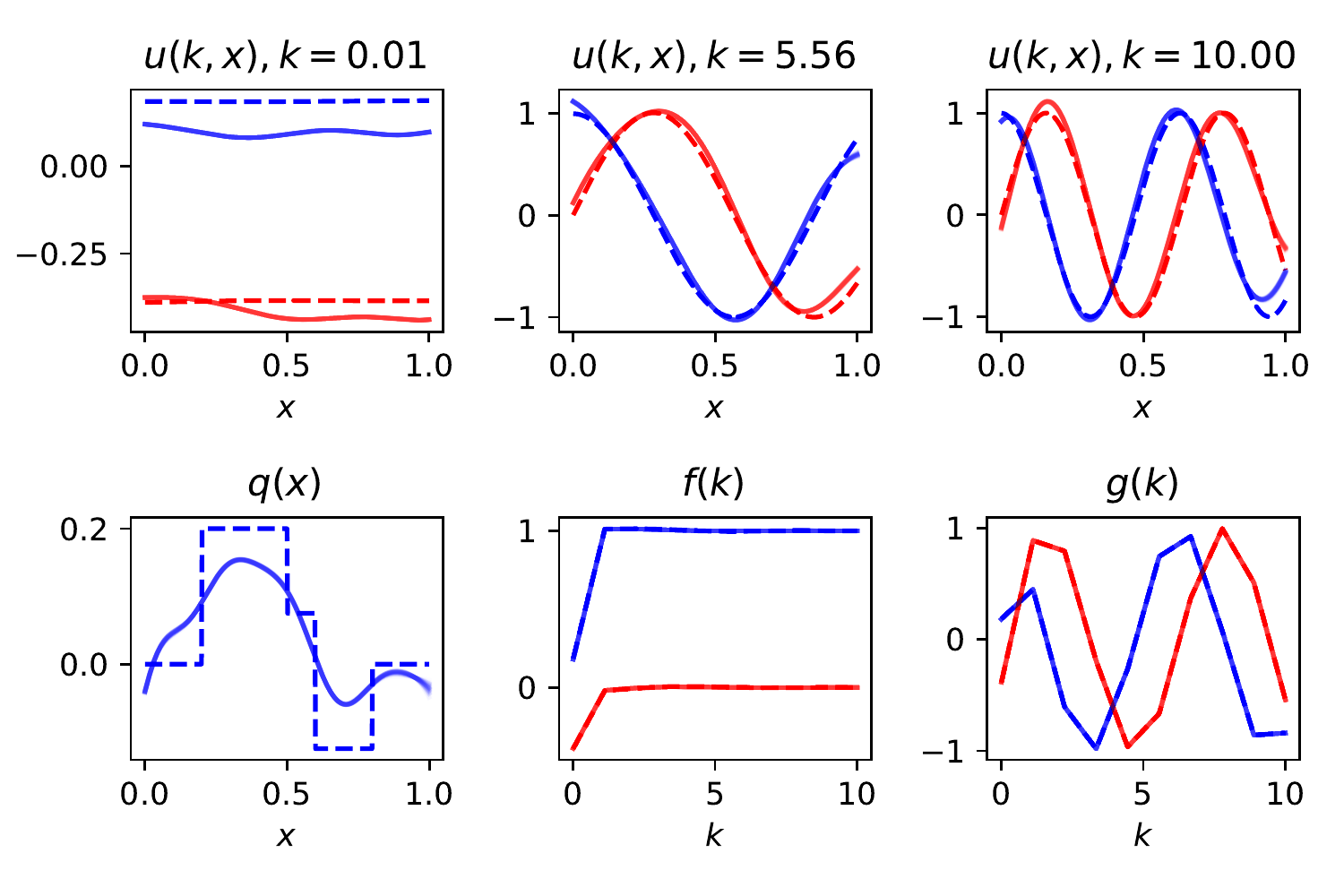} & \includegraphics[scale=.4]{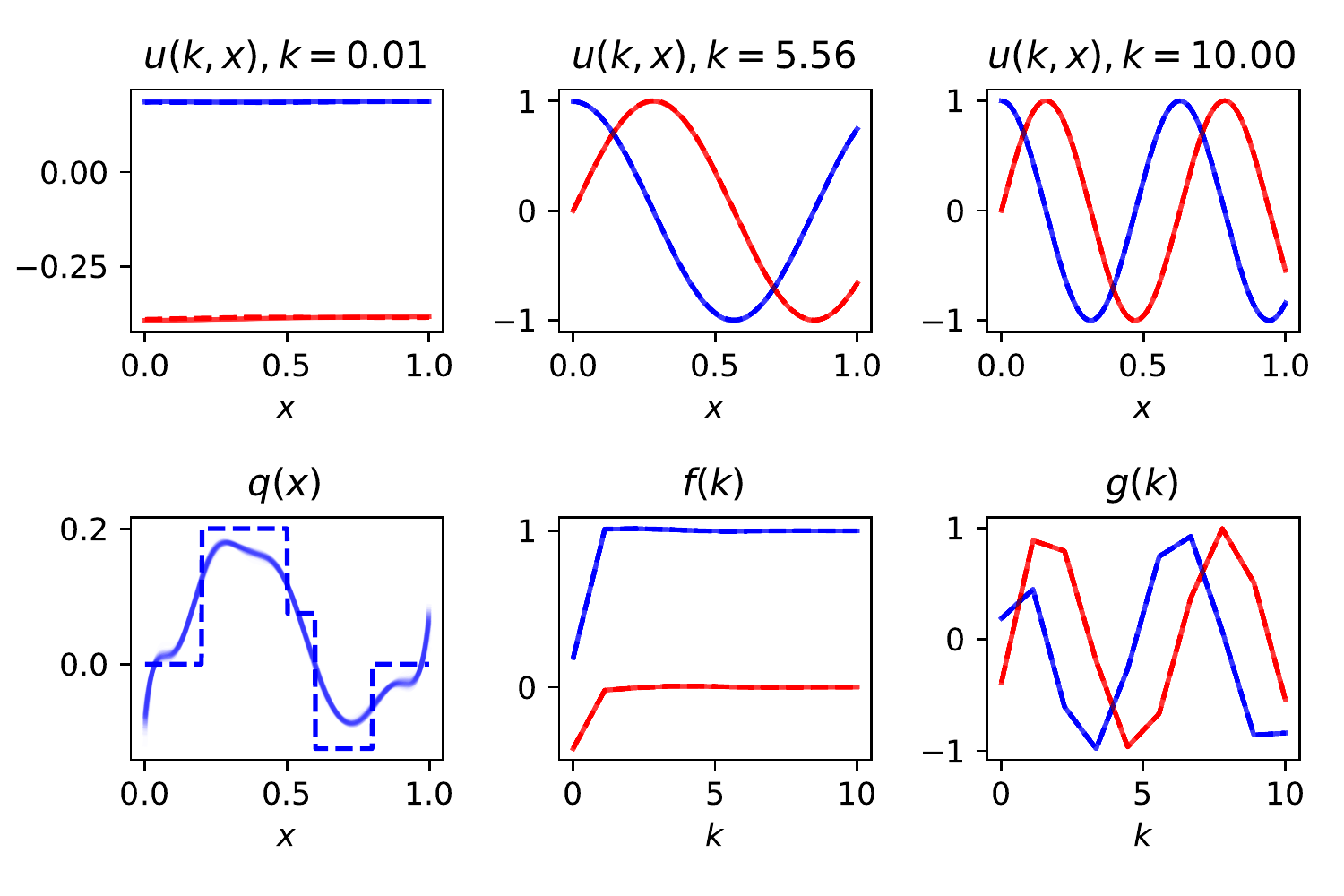}\\
\hline
\includegraphics[scale=.4]{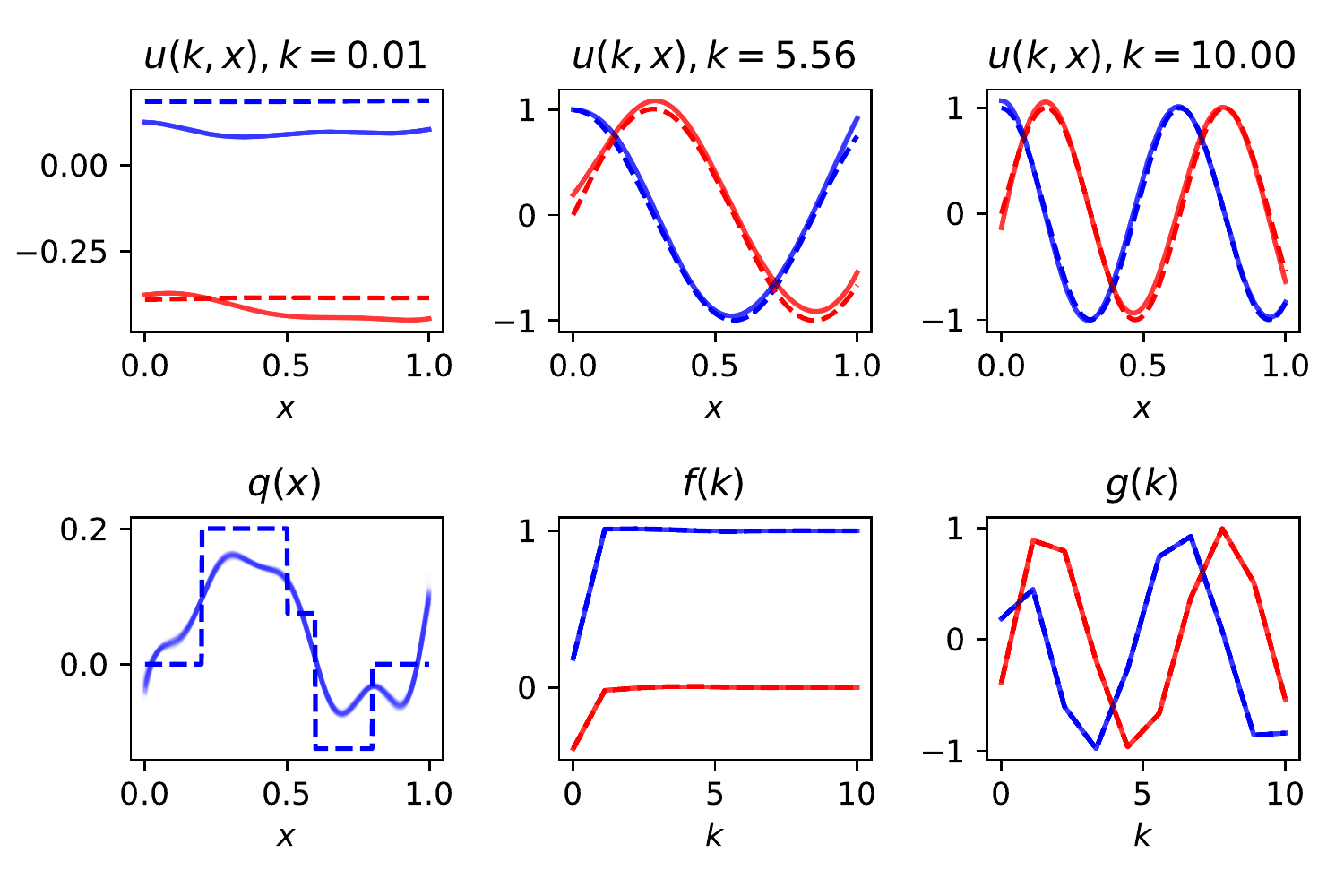} & \includegraphics[scale=.4]{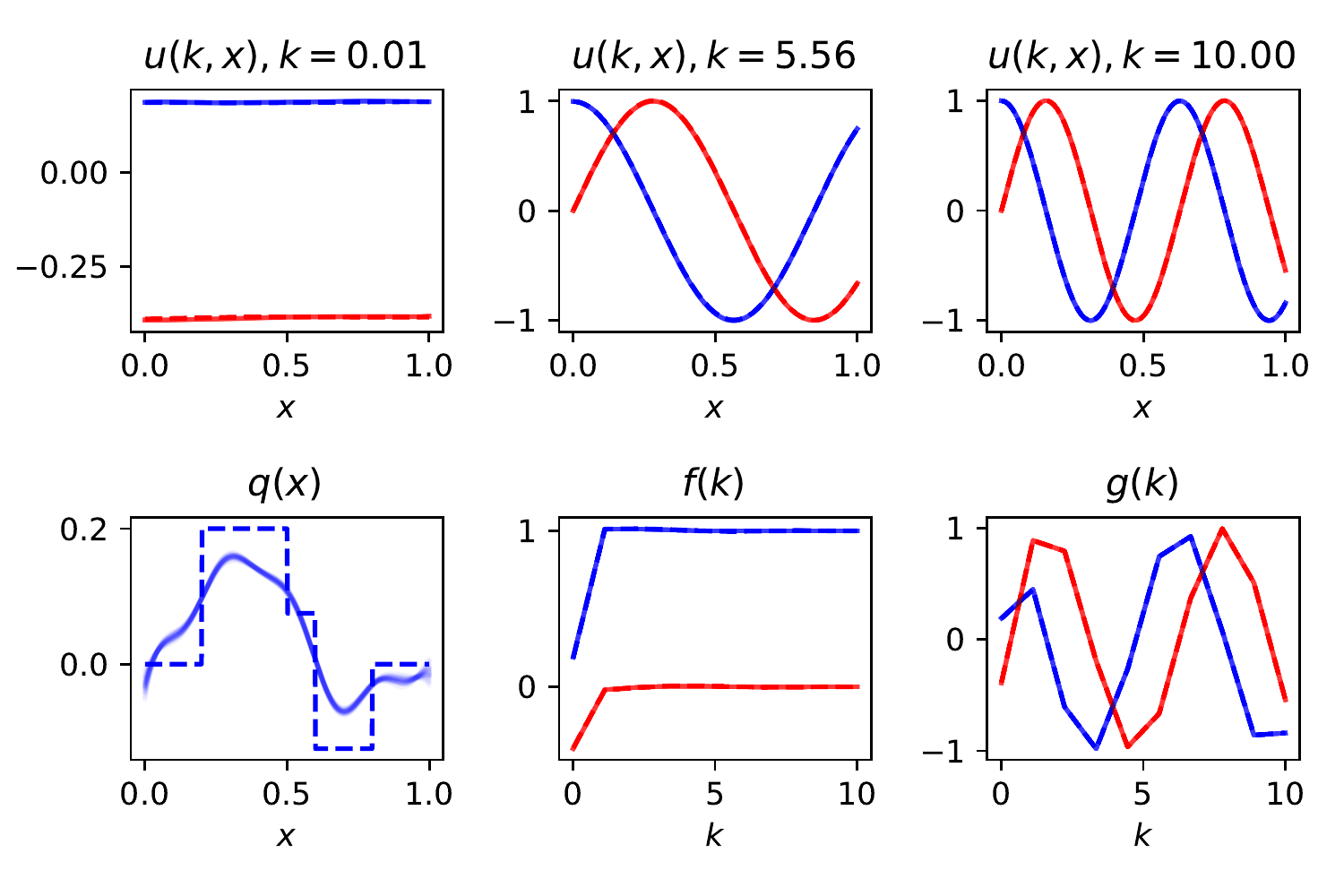}\\
\hline
\includegraphics[scale=.4]{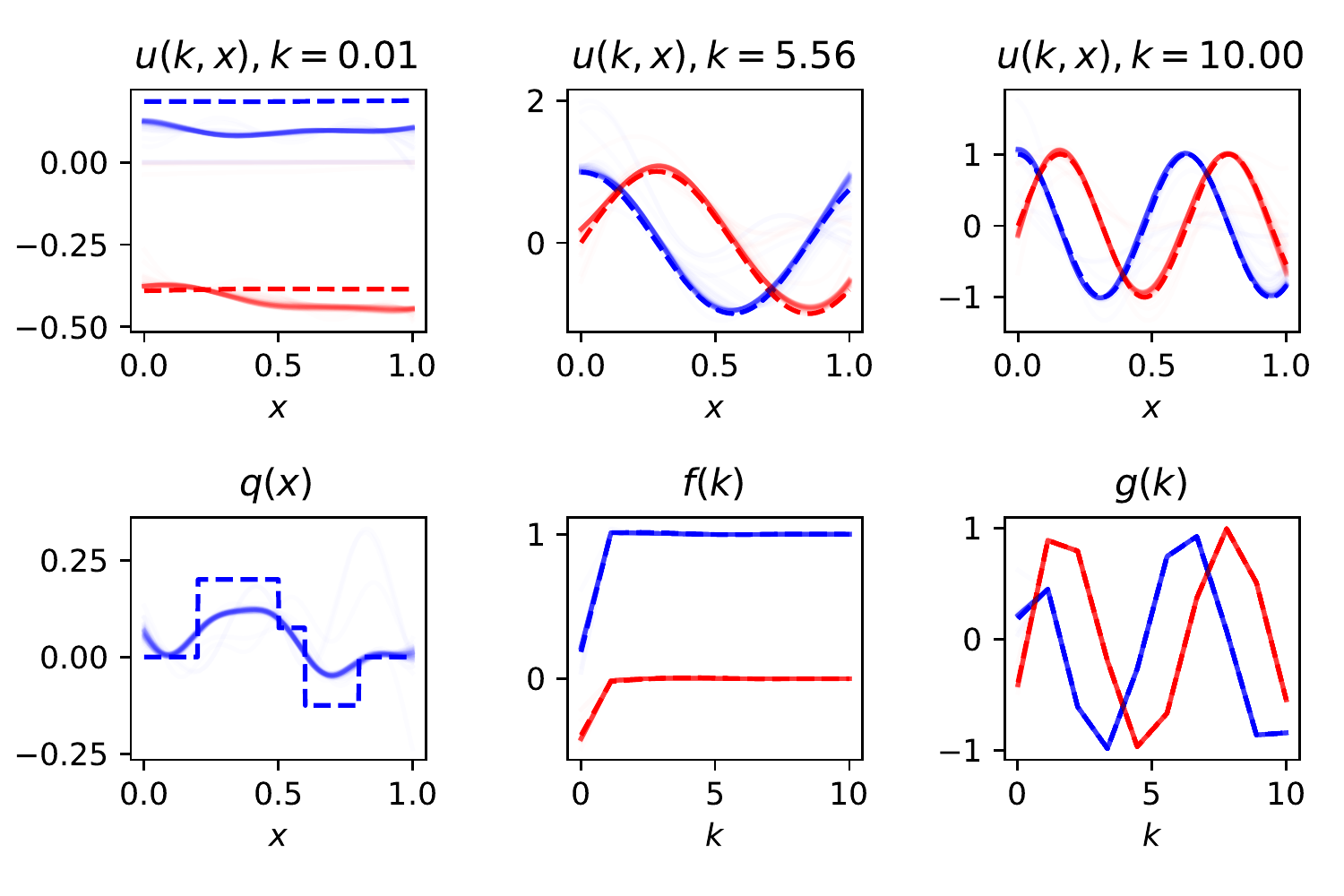} & \includegraphics[scale=.4]{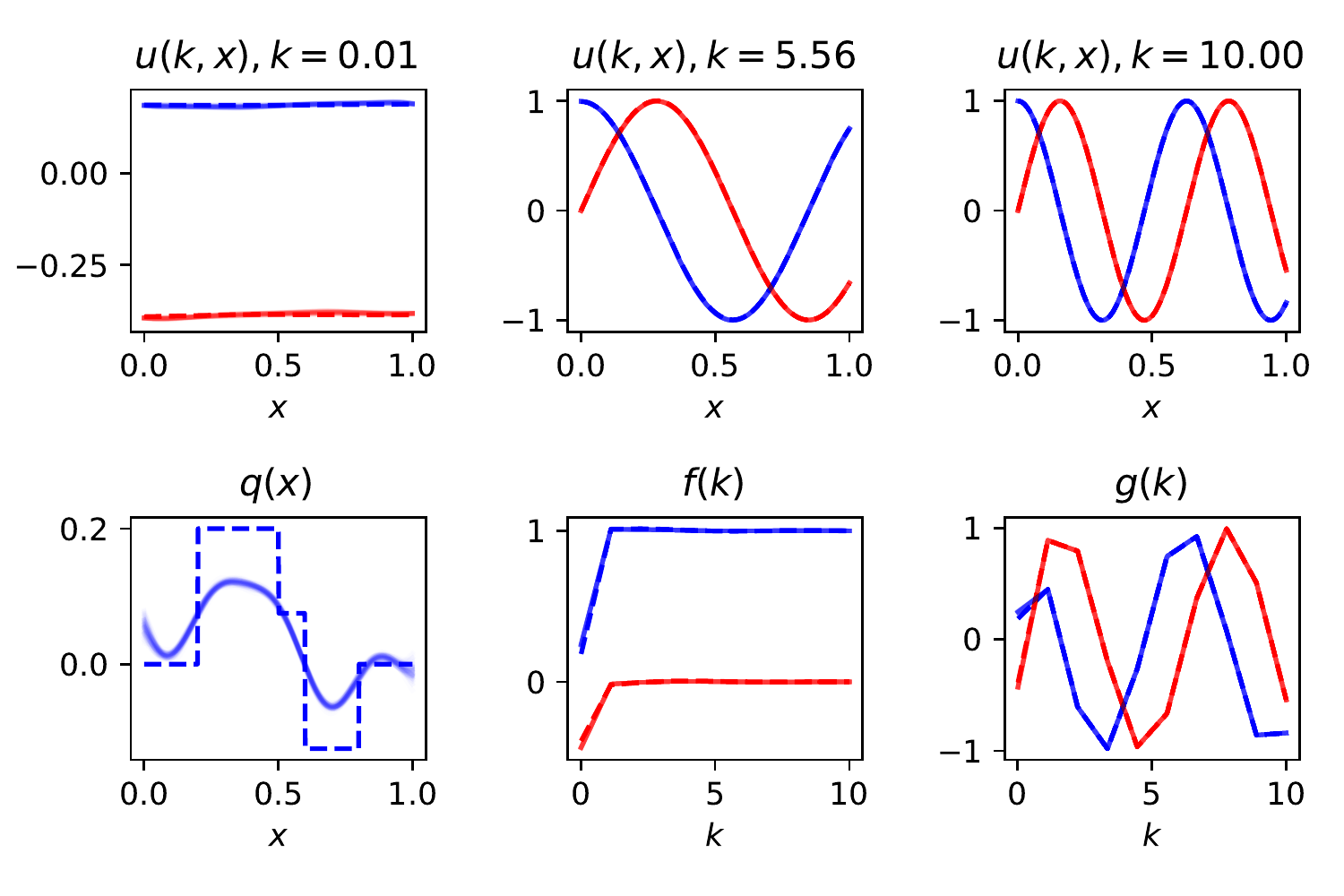}\\
\hline
\includegraphics[scale=.4]{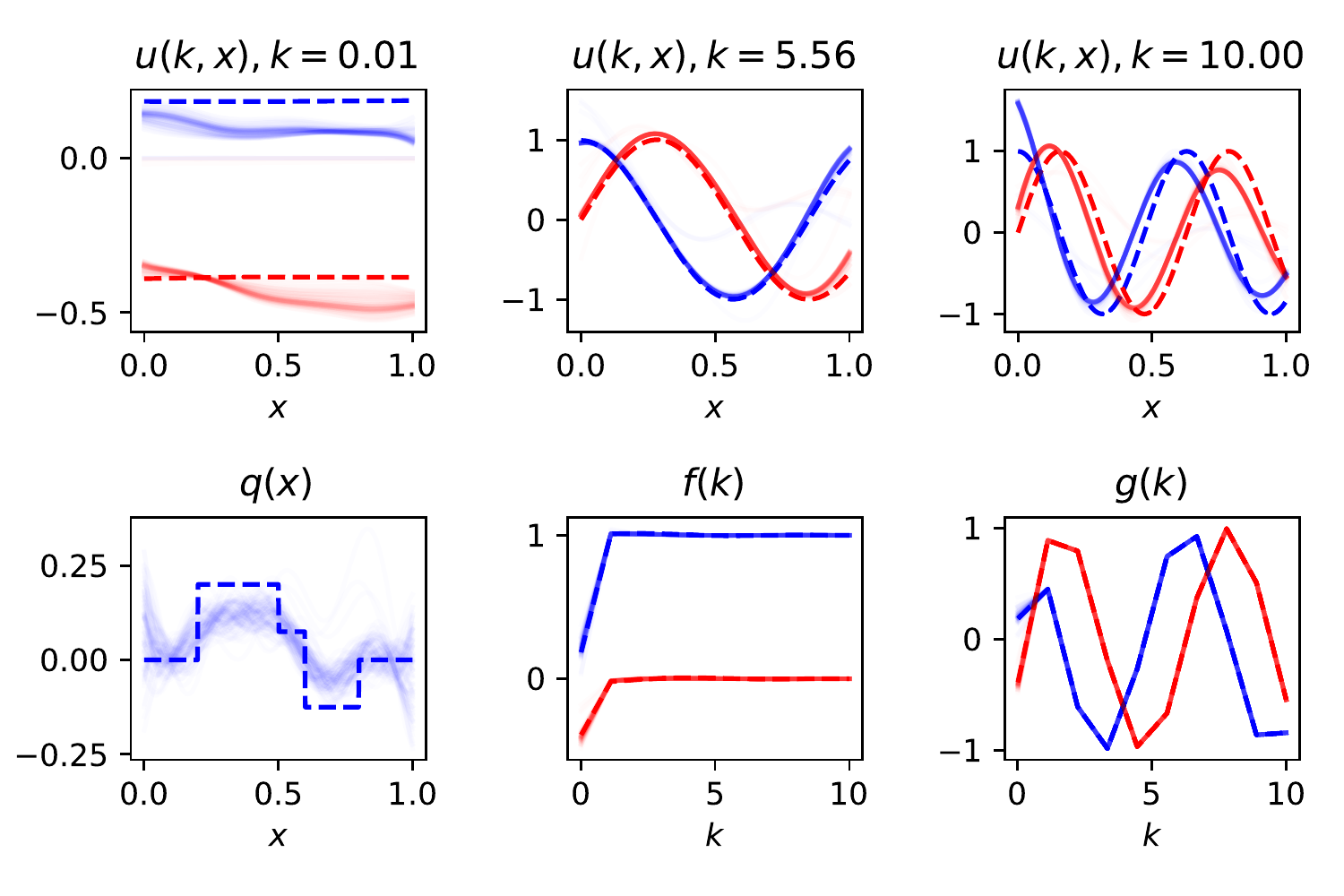} & \includegraphics[scale=.4]{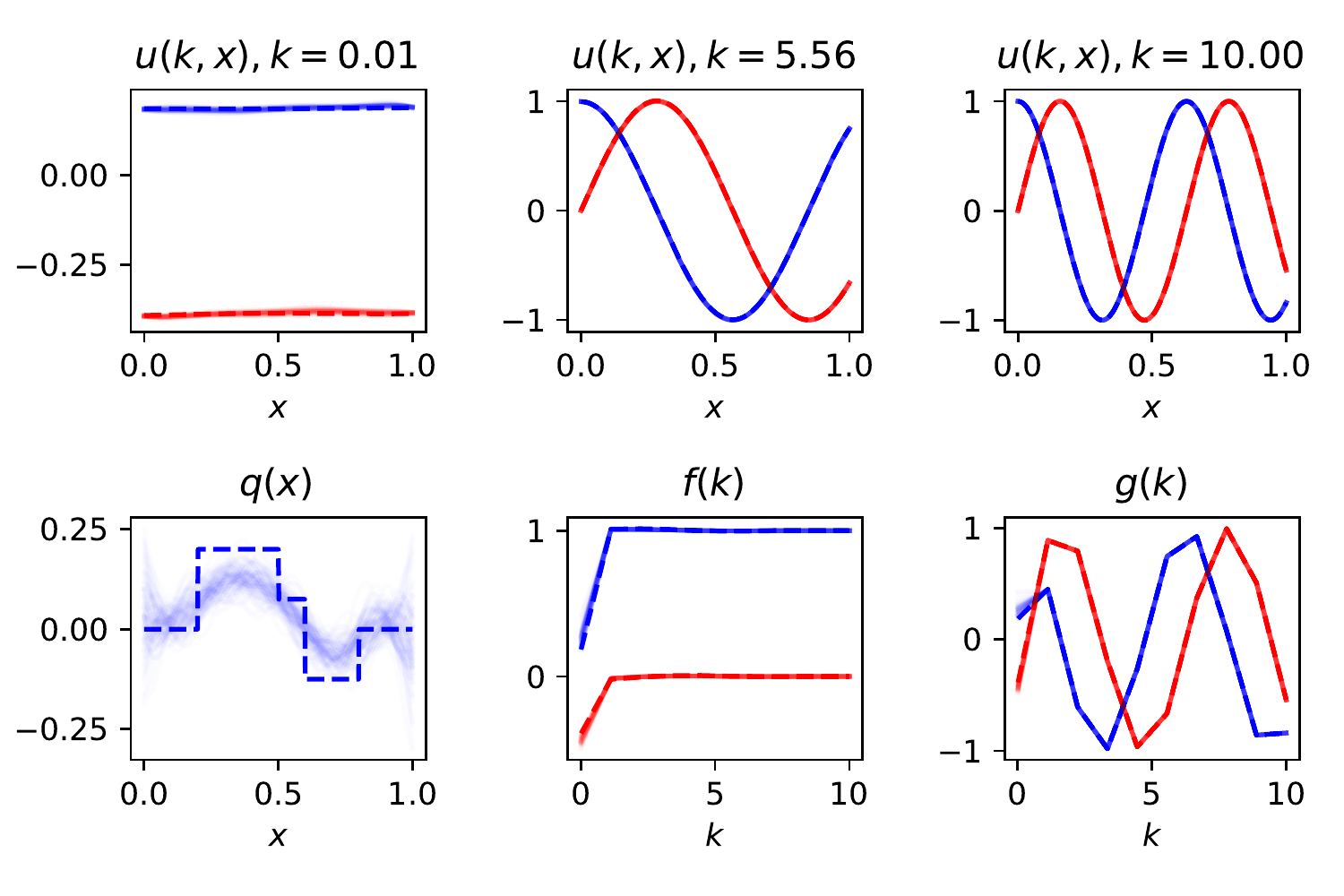}\\
\end{tabular}
\caption{Results for the LO (left) and DA (right) methods for varying noise levels ($\sigma = 10^{-6}, 10^{-5}, 10^{-4}, 10^{-3}$ respectively from top to bottom). The subplots follow the same layout as the previous figures. Individual results for different realizations of the noise are superimposed to clearly show the variation.}
\label{fig:noisyresults}
\end{figure}
\clearpage
\section{Discussion and Conclusion}
We treat the inverse problem of retrieving the scattering potential in a 1D Schrödinger equation from boundary data. To do this, we propose a two-step approach inspired by a previously-published ROM-based method. We extend this method, previously applied to 1D diffusion problems with Neumann boundary conditions, to the 1D Schrödinger equation with impedance boundary conditions. In particular, we presented explicit expressions for retrieving the ROM-matrices from boundary data and proposed a novel approach for approximating the state from these matrices. This approach, based on ideas from data-assimilation, is an alternative to the previously proposed method based on Lanczos-orthogonalization. Given the estimates of the states, the scattering potential is obtained by solving an integral equation.

We compared the two approaches numerically on a simulated example with varying noise levels. These experiments suggest that the data-assimilation approach for estimating the state is more accurate and stable and leads to a more stable estimate of the scattering potential for moderate noise levels.

This work is the first step towards extending the ROM-based approach to frequency-domain wave-like problems (e.g., the Helmholtz equation) and 2D/3D. Other open questions for further research include the approximation error, stability estimates, and more practical aspects such an iterative approach where the reference potential is iteratively updated.

\section*{Acknowledgements}
This work was supported by the Utrecht Consortium for Subsurface Imaging (UCSI).
\clearpage
\appendix 
\section{Proofs}
\begin{proof}[Proof of Lemma \ref{lma:rommatrices}]
From the weak form we find
$$S_{ij} - k_j^2 M_{ij} - \imath k_j B_{ij} = -2\imath k_j \overline{f}_i,$$
and
$$S_{ji} - k_i^2 M_{ji} - \imath k_i B_{ji} = -2\imath k_i \overline{f}_j,$$
from which (by taking the conjugate transpose and using the fact that the matrices involved are Hermitian)
$$S_{ij} - k_i^2 M_{ij} + \imath k_i B_{ij} = 2 \imath k_i f_j.$$
Combining these yields
$$(k_i^2 - k_j^2)M_{ij} - \imath (k_i + k_j)B_{ij} = -2 \imath (k_i {f}_j + k_j \overline{f}_i),$$
and
$$(k_i^2 - k_j^2)S_{ij} - \imath (k_j^2k_i + k_i^2k_j)B_{ij} =  -2 \imath (k_j^2k_i f_j + k_i^2k_j\overline{f_i}),$$
from which we can compute $M_{ij}$ and $S_{ij}$:
$$M_{ij} = \imath \left(\frac{B_{ij}}{k_i - k_j} - 2\frac{k_i {f}_j + k_j \overline{f}_i}{k_i^2 - k_j^2}\right).$$
$$S_{ij} = \imath \left(\frac{k_i k_j B_{ij}}{k_i - k_j} -2 \frac{k_j^2k_i f_j + k_i^2k_j \overline{f}_i}{k_i^2 - k_j^2}\right).$$
For the diagonal elements we need to take a limit of the above two relations.
We first compute the diagonal elements of $M$. We set $\lambda=k_j^2$, and $k^2_i=\lambda+h$.
We also define $f(k_j)=\phi(\lambda)=\phi_1+\imath \phi_2$ and $\phi(\lambda+h)=\phi_1^h+\imath \phi_2^h$
and similarly $ \gamma(\lambda)=g(k_j)$. 
Since $\Im(M_{jj})=0$, we obtain
\begin{equation*}
   M_{jj}=\lim_{h
   \to 0} \Big\{ -2\frac{\sqrt{
    \lambda} \phi_2^h-\sqrt{\lambda+h}\phi_2}{h}-\frac{\gamma_2\gamma_1^h-\gamma_1\gamma_2^h +\phi_2\phi_1^h-\phi_1\phi_2^h}{\sqrt{\lambda+h}-\sqrt{\lambda}}\Big\}=
\end{equation*}
\begin{equation*}
    -2 \Big(\sqrt{\lambda} \frac{d\phi_2}{d\lambda}(\lambda)-\frac{1}{2}\lambda^{-1/2}\phi_2(\lambda)\Big)- \gamma_2(\lambda) 2
    \sqrt{\lambda } \frac{d\gamma_1}{d
    \lambda}(\lambda)+\gamma_1(\lambda)2\sqrt{\lambda}\frac{d\gamma_2(\lambda)}{d\lambda}-
\end{equation*}
\begin{equation}
    \phi_2(\lambda) 2
    \sqrt{\lambda } \frac{d\phi_1}{d
    \lambda}(\lambda)+\phi_1(\lambda)2\sqrt{\lambda}\frac{d\phi_2(\lambda)}{d\lambda}.
\end{equation}
The product rule gives that $\frac{d\phi}{d\lambda}=\frac{df}{dk}\frac{dk}{dk^2}=f'(k)(2k)^{-1}$. Similarly for $\gamma.$ Combining gives,
\begin{equation*}
    M_{jj}=\Big\{-2\Big(k \frac{1}{2k}\Im(f')-\frac{1}{2k}\Im(f)\Big)-
\end{equation*}
\begin{equation*}
    \Im(g) 2 k \frac{1}{2k} \Re(g)+\Re(g) 2 k \frac{1}{2k} \Im(g)-\Im(f) 2 k \frac{1}{2k} \Re(f)+\Re(f) 2 k \frac{1}{2k} \Im(f)\Big\}\Big|_{k=k_j},
\end{equation*}
which gives
\begin{equation*}
    M_{jj} = \Re(f_j)\Im(f_j') - \Im(f_j)\Re(f_j') + \Re(g_j)\Im(g_j') - \Im(g_j)\Re(g_j') - \Im(f_j') + \Im(f_j)/k_j.
\end{equation*}
We obtain similarly the relation for the diagonal of $S$.
\end{proof}
\clearpage
\section{Regularization Parameter Selection}
The LO and DA methods both have two regularization parameters that regularize the problem. These parameters are chosen to minimize the expected reconstruction error for the given noise level. We approximate the expected error by averaging the error over 100 realization of the noise. The plots corresponding to the results presented in table \ref{tabone} and figures \ref{rec_Lan}, \ref{rec_DA}, and \ref{fig:noisyresults} are shown in figure \ref{fig:reg-noise}.

\begin{figure}
\centering
\begin{tabular}{c|c}
\includegraphics[scale=.4]{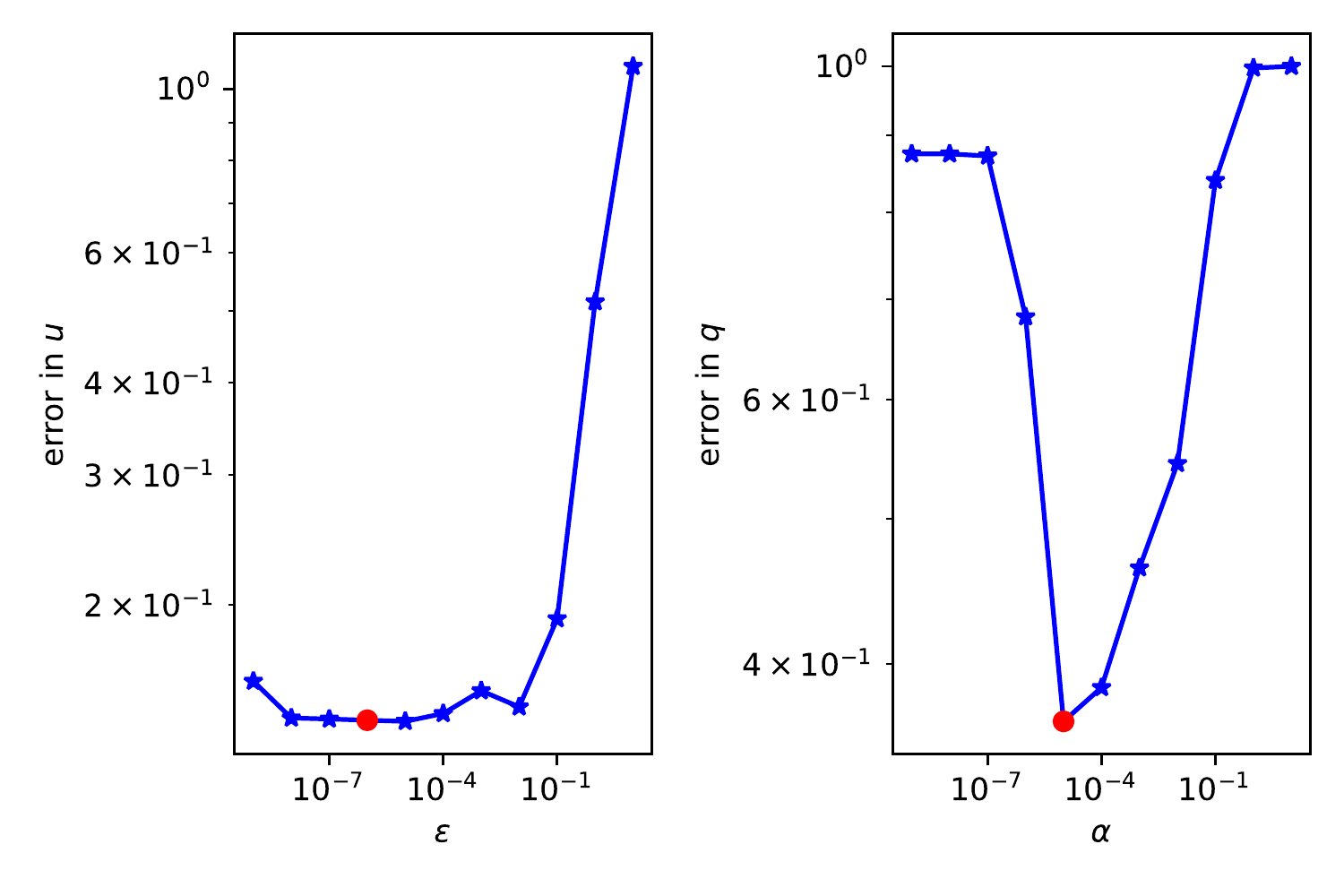} &
\includegraphics[scale=.4]{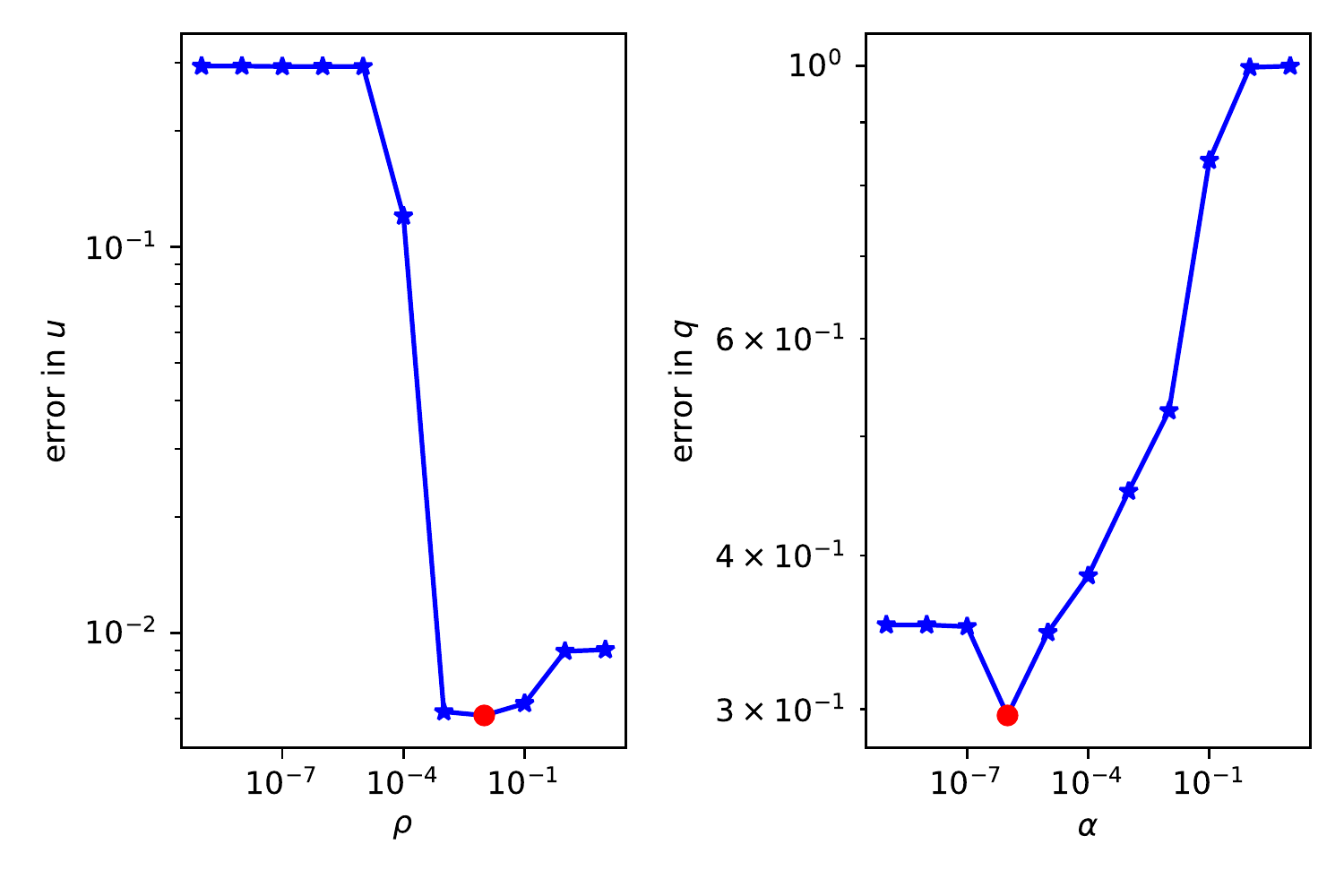} \\
\hline
\includegraphics[scale=.4]{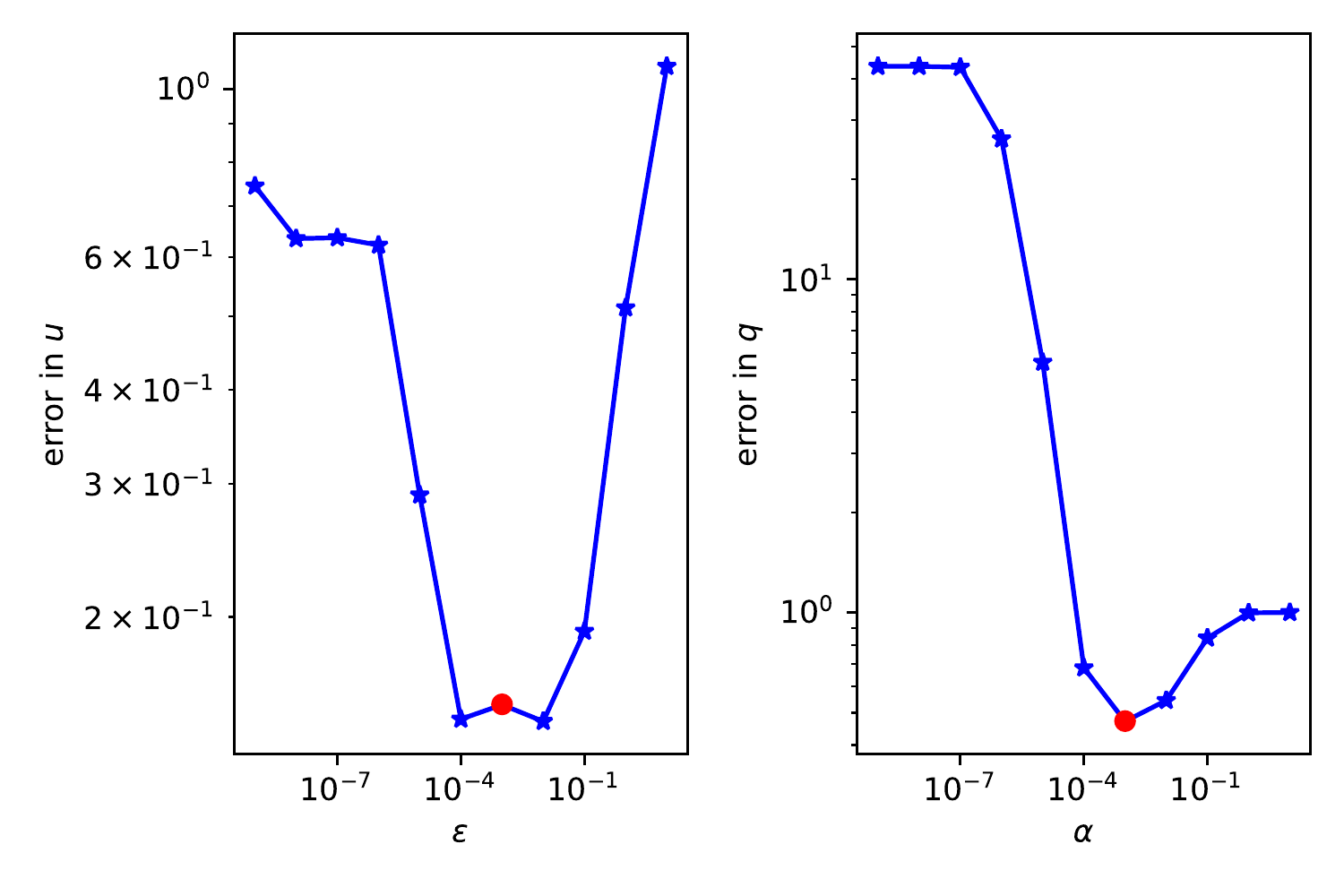} &
\includegraphics[scale=.4]{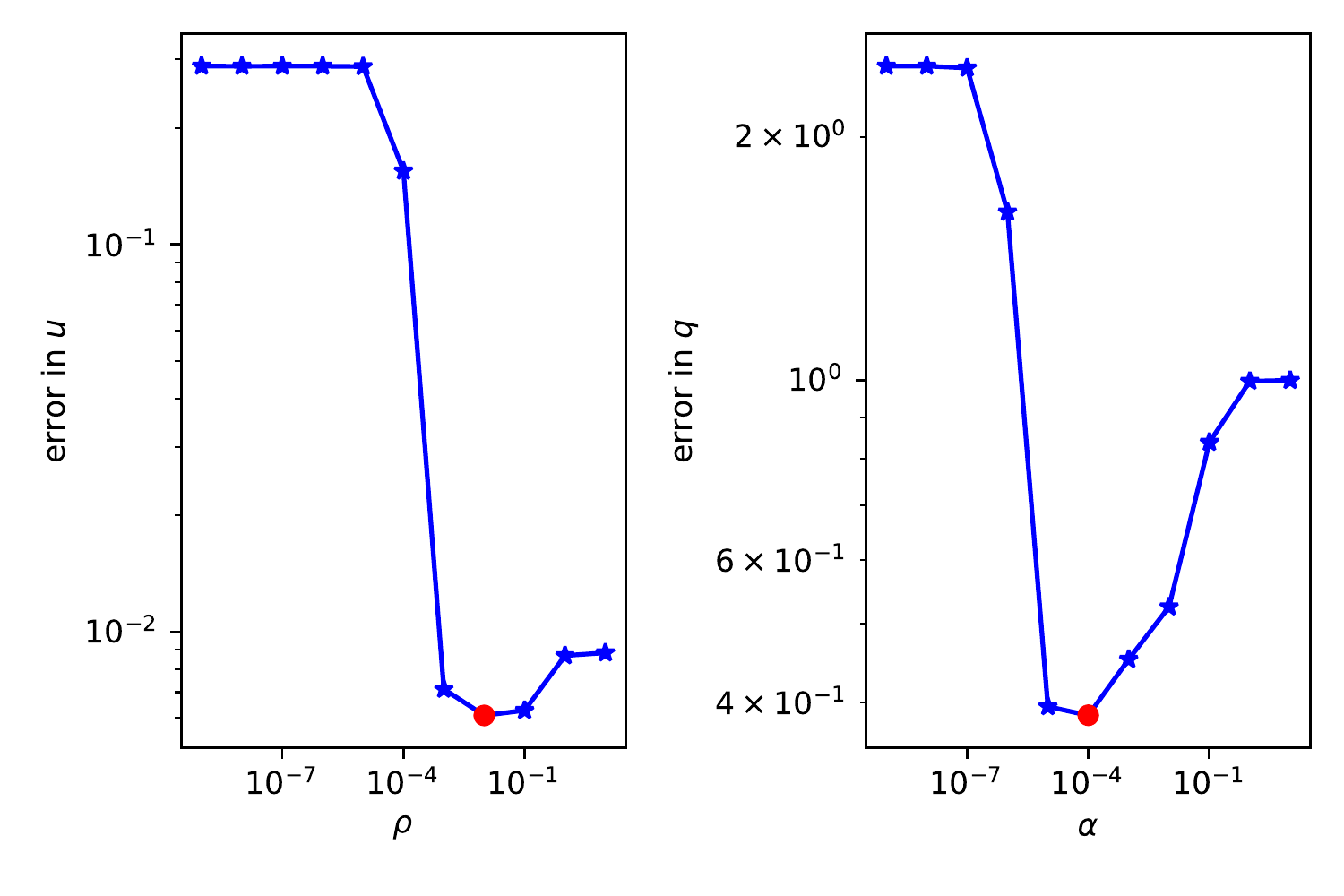} \\
\hline
\includegraphics[scale=.4]{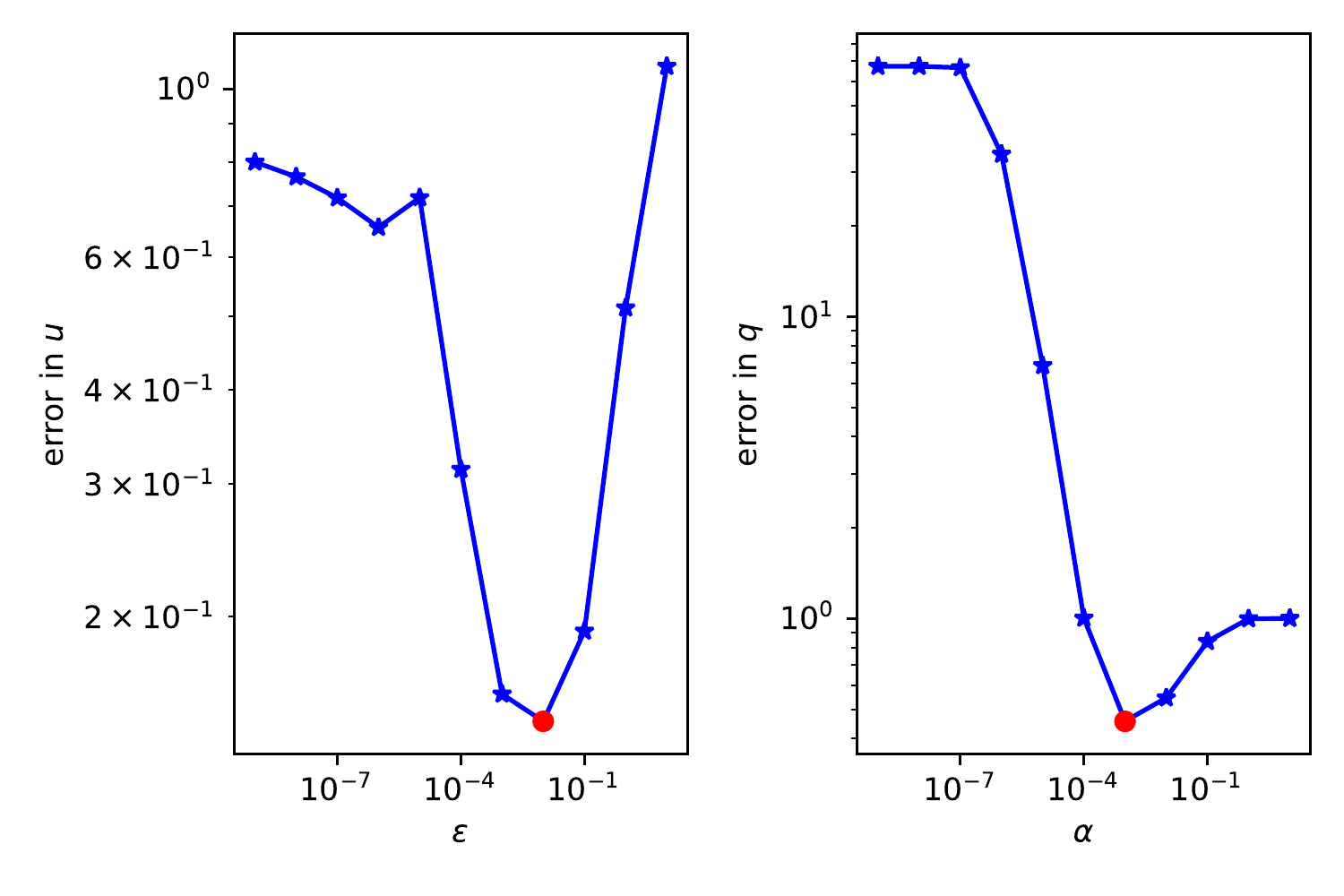} &
\includegraphics[scale=.4]{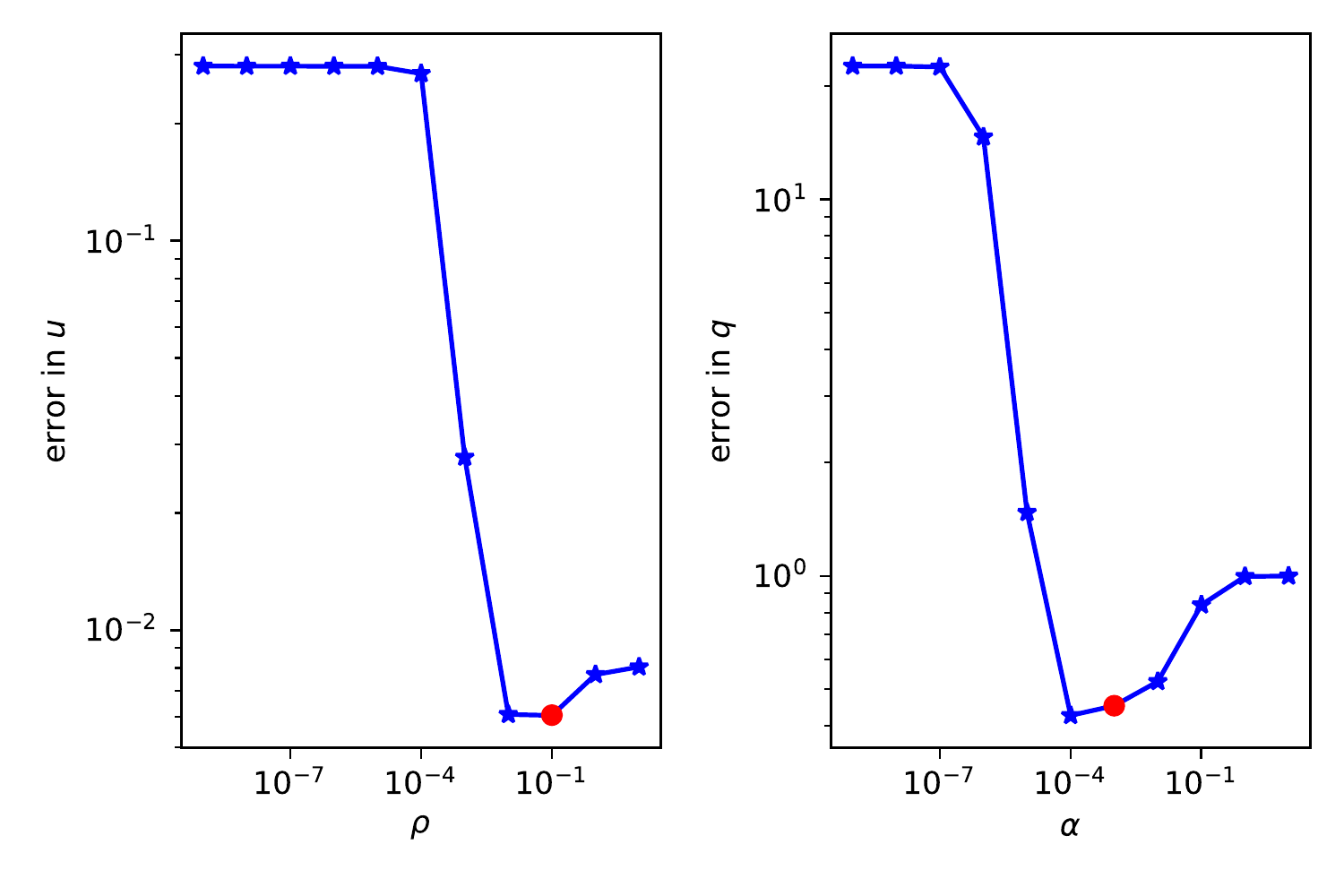} \\
\hline
\includegraphics[scale=.4]{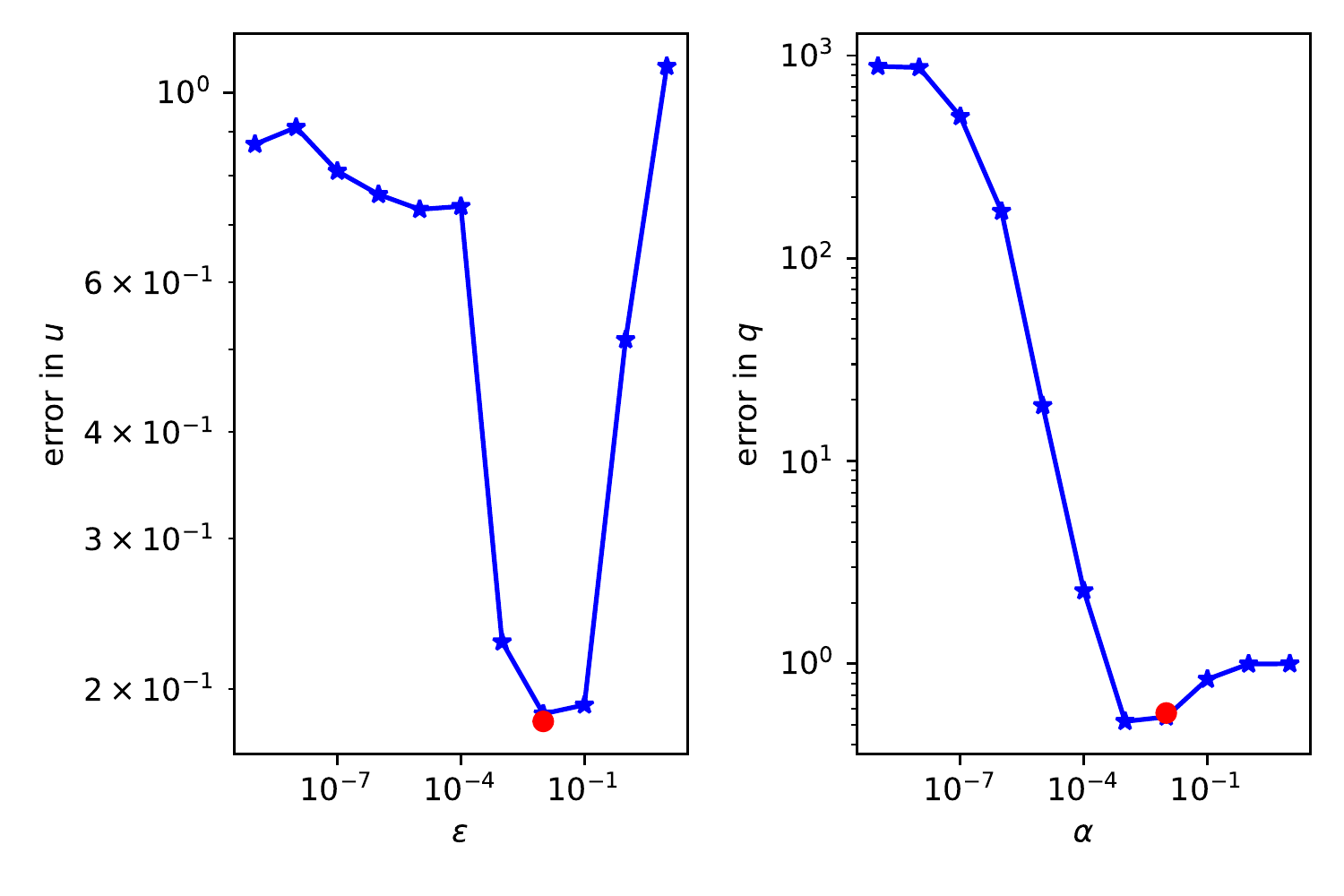} &
\includegraphics[scale=.4]{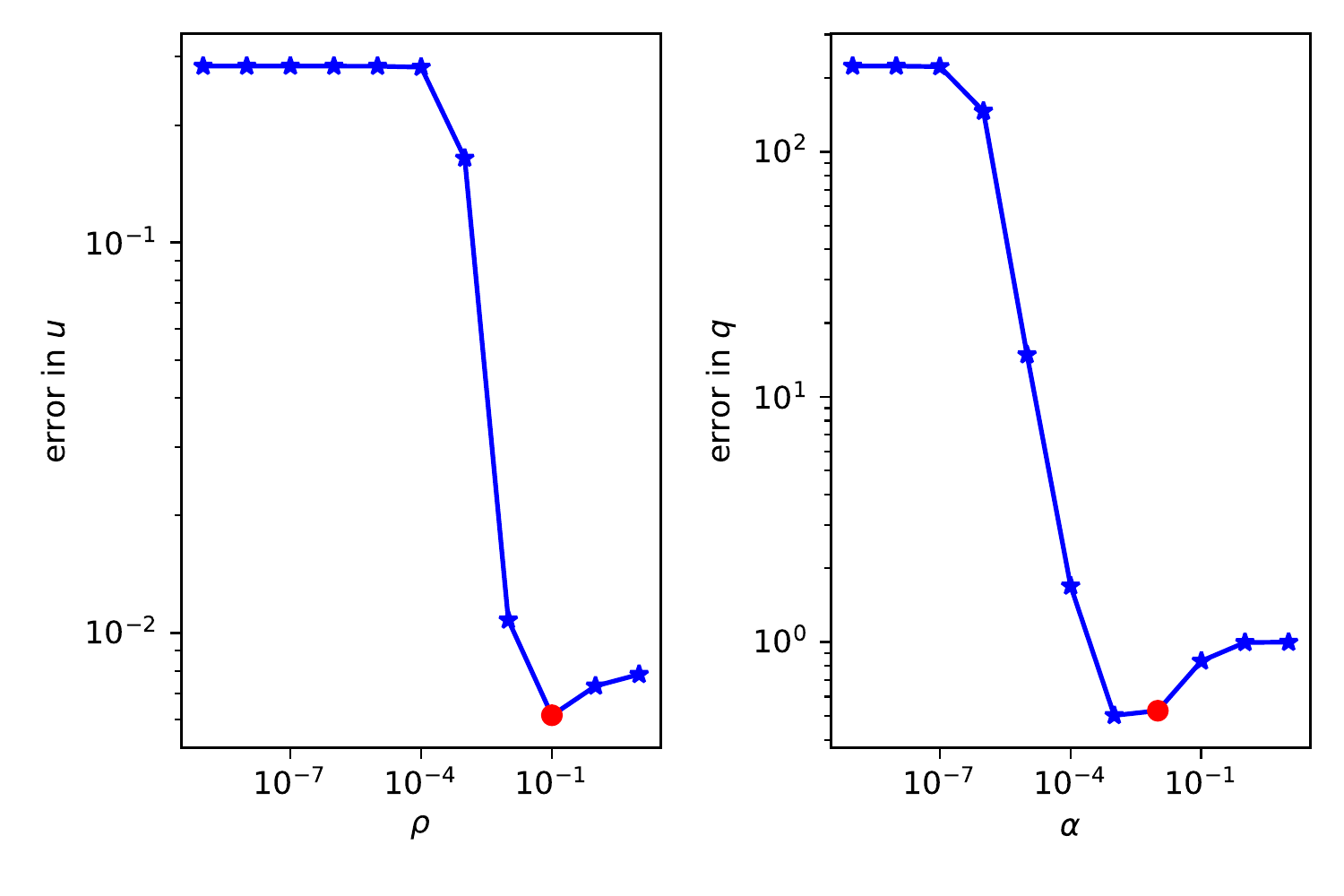} \\
\hline
\includegraphics[scale=.4]{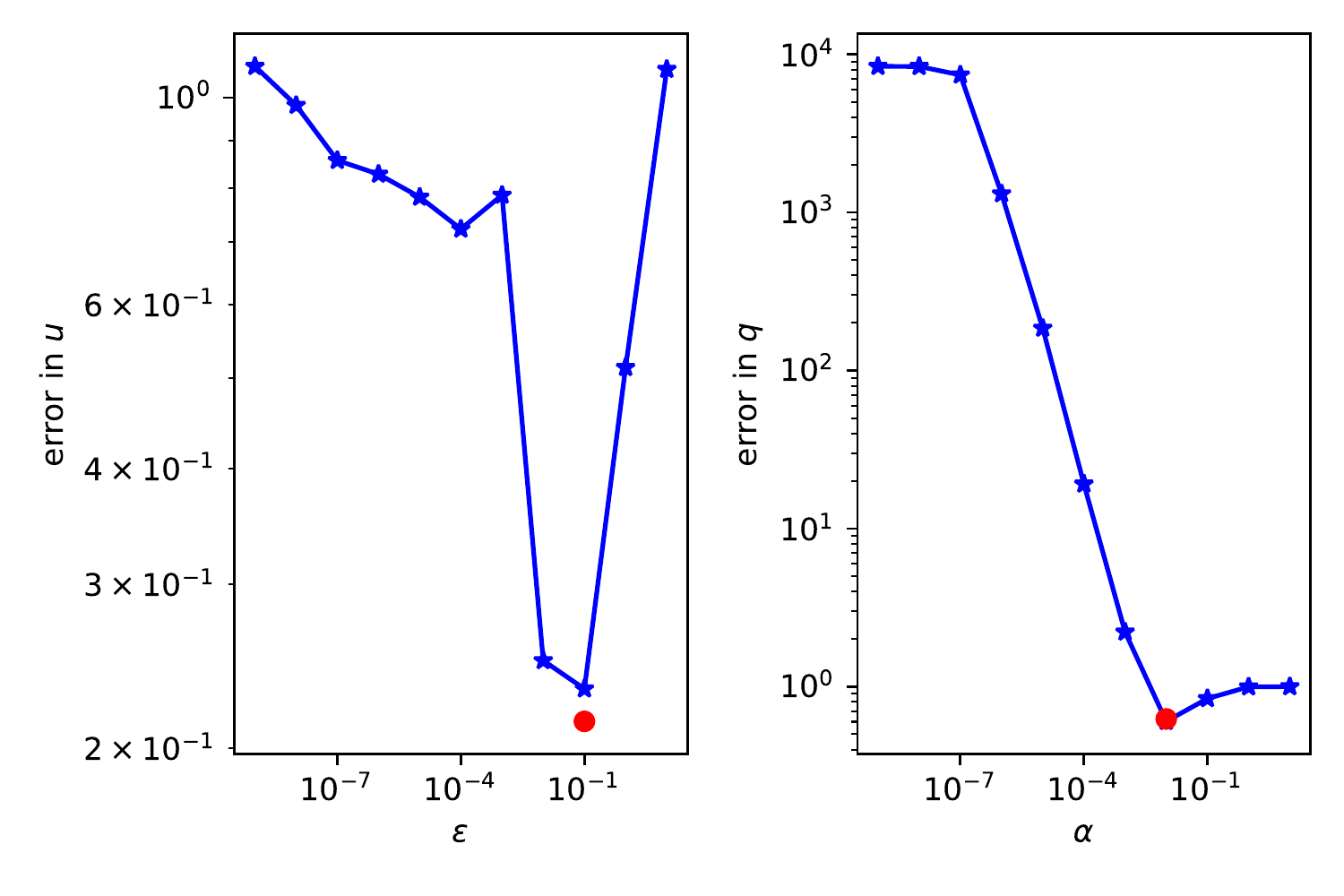} &
\includegraphics[scale=.4]{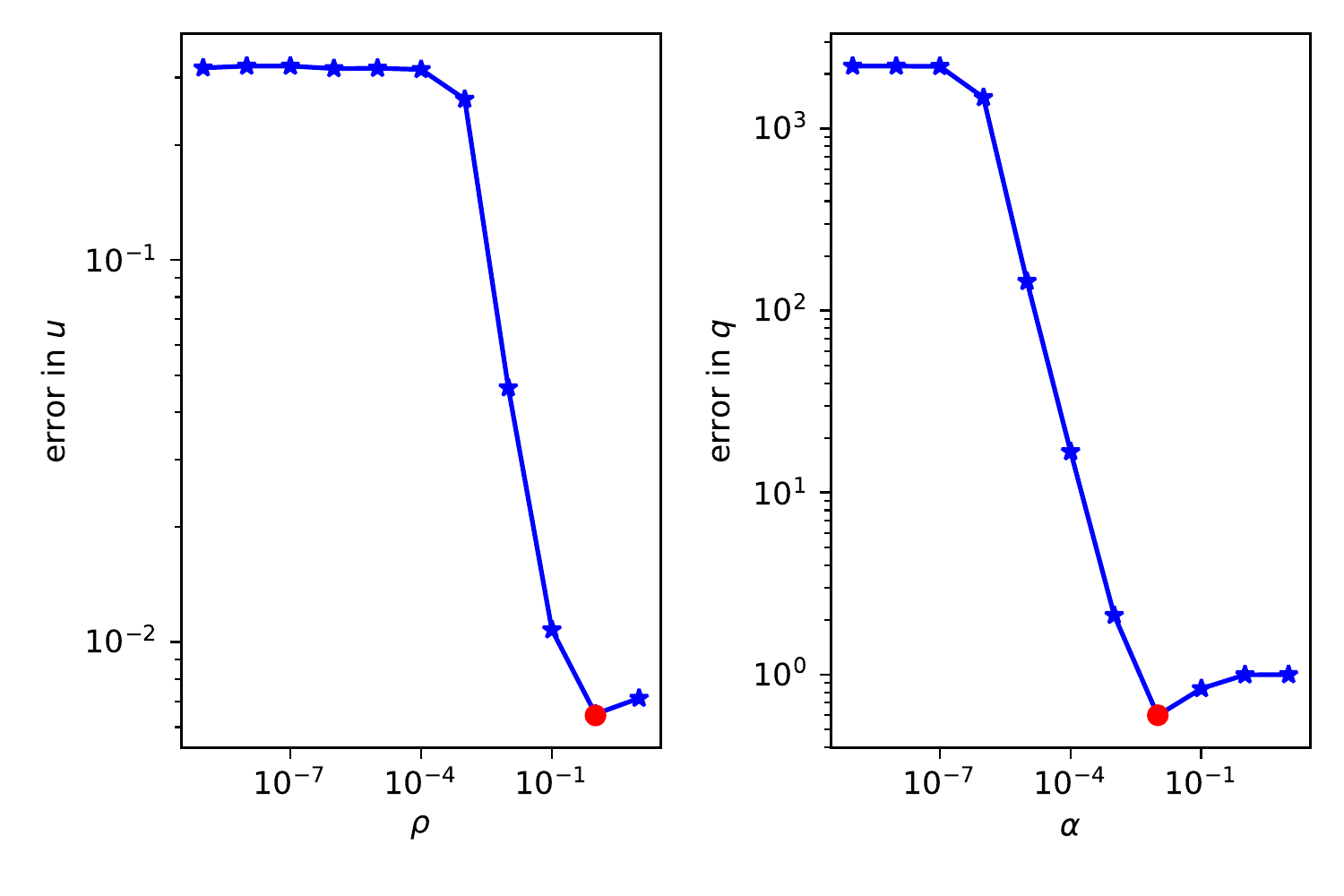} \\
\end{tabular}
\caption{Average error for both methods (LO, left and DA,right) for various noise levels ($0, 10^{-6}, 10^{-5}, 10^{-4}, 10^{-3}$ respectively from to to bottom).}
\label{fig:reg-noise}
\end{figure}
\clearpage

\end{document}